\RequirePackage{fix-cm}
\documentclass[10pt]{svjour3}                     
\smartqed  
\usepackage{graphicx}
\usepackage{amssymb,color}
\usepackage{amsmath}
\usepackage{amsfonts}
\usepackage{float}
\usepackage{graphicx}
\usepackage{subcaption}
\usepackage{mathtools}
\usepackage{dsfont}
\usepackage{float}
\newcommand{\muis}[1]{{\mu}_{#1}^{IS}}
\newcommand{\muisn}[1]{{\mu}_{#1}^{(n),IS}}
\newcommand{\mus}[1]{{\mu}_{#1}^{S}}
\newcommand{\musn}[1]{{\mu}_{#1}^{(n),S}}
\newcommand{\murs}[1]{{\mu}_{#1}^{RS}}
\newcommand{\mursn}[1]{{\mu}_{#1}^{(n),RS}}
\newcommand{\muvs}[1]{{\mu}_{#1}^{VS}}
\newcommand{\muvsn}[1]{{\mu}_{#1}^{(n),VS}}
\newcommand{\mun}[1]{{\mu}_{#1}^{(n)}}

\newcommand{\gxx}[1]{\partial_{\alpha\alpha} g(\alpha_{#1},\theta_{#1})}
\newcommand{\gx}[1]{\partial_\alpha g(\alpha_{#1},\theta_{#1})}

\newcommand{\gxy}[1]{\partial_{\alpha\theta} g(\alpha_{#1},\theta_{#1})}

\newcommand{\omu}[2]{{\mu}^{#1}_{#2}}

\newcommand{\ninf}{\Vert f \Vert_\infty}
\newcommand{\sumn}[1]{\sum_{ #1 \in \N}}
\newcommand{\E}[1]{\mathbb{E}\left[ #1 \right]}

\newcommand{\p}[2]{{p}_{#2}^{#1}}
\newcommand{\pint}[2]{\langle {#1},{#2}\rangle}
\newcommand{\U}{\mathcal{U}}
\newcommand{\N}{\mathbb{N}}
\newcommand{\D}{\mathcal{D}}
\newcommand{\B}{\mathcal{B}}
\newcommand{\C}{\mathcal{C}}
\newcommand{\R}{\mathbb{R}}
\newcommand{\M}{\mathcal{M}_F(\mathbb{N}_0)}

\newcommand{\pcomun}{p^n_{t}(j,l,m\mid k-1)}
\newcommand{\ptilde}{\tilde{p}^n_{t}(j,l,m\mid k-1)}
\newcommand{\qcomun}{q^n_{t}(j,l,m\mid k)}
\newcommand{\qtilde}{\tilde{q}^n_{t}(j,l,m\mid k)}
\newcommand{\mea}{\mathcal{M}_{\varepsilon,A}}
\newcommand{\moa}{\mathcal{M}_{0,A}}

\begin{document}

\title{SIR dynamics with Vaccination in a large Configuration Model}


\author{Emanuel Javier Ferreyra    \and
        Matthieu Jonckheere \and
        Juan Pablo Pinasco
}





\institute{Emanuel Javier Ferreyra \at
              Instituto de C{\'a}lculo UBA-CONICET, \\
              Facultad de Ciencias Exactas y Naturales,
              Universidad de Buenos Aires, \\
              Av Cantilo s/n, Ciudad Universitaria(1428) Buenos Aires, Argentina.\\
              \email{emanueljf@gmail.com}           
           \and
           M. Jonckheere \at
              Instituto de C{\'a}lculo UBA-CONICET, \\
              Facultad de Ciencias Exactas y Naturales,
              Universidad de Buenos Aires, \\
              Av Cantilo s/n, Ciudad Universitaria(1428) Buenos Aires, Argentina.\\
              \email{matthieu.jonckheere@gmail.com}
           \and
           J.P. Pinasco \at
              IMAS UBA-CONICET  and  Departamento  de Matem{\'a}tica, \\
              Facultad de Ciencias Exactas y Naturales,
              Universidad de Buenos Aires, \\
              Av Cantilo s/n, Ciudad Universitaria(1428) Buenos Aires, Argentina.\\
              \email{jpinasco@gmail.com}}

\date{Received: date / Accepted: date}

\maketitle

\begin{abstract}
We consider a  SIR model with vaccination strategy  on a sparse configuration
model random graph. We show the convergence of the system when the number of nodes
grows  and characterize the scaling limits. Then, we prove the existence of  optimal controls
for the limiting equations formulated in the framework of game theory, both in the
centralized and decentralized setting.

We show how the characteristics of the graph (degree distribution) influence the
vaccination efficiency for optimal strategies, and we compute the limiting final
size of the epidemic depending on the degree distribution of the graph and the parameters
of infection, recovery and vaccination. We also present several simulations for two types of vaccination, showing how the optimal controls allow to decrease the number of infections and underlining the crucial role of the network characteristics in the propagation of the disease and the vaccination program.

\keywords{SIR-V \and epidemics \and configuration model \and optimal control}
\end{abstract}

\section{Introduction}
\label{intro}

While epidemic dynamics have been studied extensively in the context
of mean-field models where each individual potentially interacts with every other individual,
 there has been more recently a research effort to include the
effect of local interactions  using a sparse random graph, see
\cite{kiss2017mathematics,may2001infection,moreno2002epidemic}. The non-homogeneity
in the inter-individual interactions can be reflected in the model by defining a
degree-distribution describing the statistics of the interactions, see also  the complete and
pedagogical review \cite{RevModPhys} which contains both historical and modern references
on the epidemic-like processes on complex networks.

However, a rigorous mathematical description of epidemics on networks is a challenging
problem, since a mean field approach for individuals with the same number of contacts implies
to consider infinitely many systems of differential equations of SIR
(susceptible - infectious - removed/recovered) or SIS (susceptible - infectious - susceptible) type.
These systems are coupled by the distribution of links between nodes with $k$ and $j$
neighbors, for any pair of values $k$, $j\in \N$. A seminal work was
\cite{decreusefond2012large}, where a set of finitely many  differential equations
describing the asymptotic of a SIR dynamics on a sparse Configuration Model has been
rigorously derived as projections of the infinite dimensional system.  This is a remarkable
result as it allows the model to grasp both the specifics of the interaction graph and the
epidemic dynamics in a simple finite dimensional deterministic dynamic system.  Their work
agrees with other approximations in the literature, like the work of Volz which describe a
Poissonian SIR epidemics using coupled non-linear ordinary differential equations
\cite{volz2008sir}, and shows that these equations are indeed verified in the thermodynamic
limit (i.e., when the number of nodes tends to infinity in a sparse Configuration Model). See
also \cite{janson2014} for more on the SIR dynamics on the configuration model and
\cite{miller2011note} for equivalent formulation of the ODE dynamics.

\bigskip
On the other hand, vaccination processes were extensively studied in last years, since the
anti-vaccine movements threaten social health programs by playing a common goods
dilemma: they try to avoid the individual costs of vaccination and simultaneously pretend to
enjoy the advantages of herd immunity. By modeling vaccination as a game, we are faced
with the classical difference between individual and social optima, and worse equilibria are
reached due to individual actions, than the ones obtained by a centralized planner.

So, the optimal vaccination problem has been studied  using control  and game theory tools,
see
\cite{anderson1992infectious,galvani2007long,lu2002effect,TAKEUCHI200639,zaman2008stability}.
Two main points of view are considered: a rational individual immersed in the population
who maximize its own benefit, or a centralized agent who takes decisions for the overall
population, for example a government.

The optimization problem for a centralized agent can be thought of as to minimize the costs
of a vaccination program while preventing the epidemic spread of a disease, for example,
achieving herd immunity. In \cite{HETHCOTE1973365} the authors consider a deterministic
epidemic compartmental model and use a time-dependent vaccination rate and linear costs
to get the optimal vaccination strategy.

Let us observe that  the authors in \cite{fu2010imitation} propose an evolutionary
game-theoretic problem, where individuals use evidence to estimate costs of vaccination,
and the model is based on the agent point of view.  Vaccination strategy can be also
considered influenced by a neighborhood behavior or it may depend on the individual's
beliefs about their neighborhood vaccination strategies \cite{mbah2012impact,pires2017dynamics}.
Another approach can be found in \cite{galvani2007long}, where they study how the
psychology of individuals intervenes in their perception of their risk, susceptibility or
mortality rates.

Regarding a network background, the authors in \cite{yang2019efficient} show that the
vaccination is most effective when the full network structure is known by the individual
agents although in real world  the decisions are based on partial information of the contact
underlying graph. This could justify that individuals decide to get vaccinated with rates that
depend on their degree in the network.  Highly-connected individuals (hubs) have a high
incentive to vaccinate, whereas individuals with few contacts have less incentive to
vaccinate as exposed for example in
\cite{cornforth2011erratic,mbah2012impact,TAKEUCHI200639}.  We refer the interested
reader to  \cite{brauer2012mathematical} for an extensive review of compartmental models
for epidemic modeling, both in mean field setting and in networks mode, the discussion on the trade-off
between simple vaccination models which miss a lot of details but are very useful to reach a
general qualitative analysis, and more detailed models usually designed for quite specific
diseases and populations.

In \cite{ball2013acquaintance,britton2007graphs} the authors considered different  local vaccination strategies,
where randomly chosen individuals  (friends) of the selected individuals are vaccinated,
prior to the introduction of the disease. They obtained formulas for the final size of the
epidemic as function of the vaccination intensity using  branching processes techniques and
generating functions.

In this work, we study degree related vaccination strategy including the case of ``proportional to degree'' strategy, which is very similar to the acquaintance
vaccination model, originally proposed in \cite{cohen2003efficient}. Let us remark however that a major difference in our model is that the vaccination
process and the epidemic occur simultaneously. 
We consider a Markov-variant of the SIR model, where  infectious  periods  are
exponentially distributed, see for details and more general models \cite{barbour2013approximating}. We adapt
the techniques developed in \cite{decreusefond2012large} and \cite{janson2014} to show the
convergence of the degree measures that describe the Markovian dynamics of propagation
on the random graph in this context and obtain a generic and deterministic description of
the epidemic evolution. As in the case without vaccination, we are able to derive a
finite-dimensional differential system describing the evolution of the quantities that usually
describe the epidemics, namely the number of individuals in each compartment (Susceptible,
Infected, Recovered and Vaccinated) and the basic and effective reproduction number.

The vaccination model deals with a context where individuals do not know if their
contacts are infected or not (otherwise, they might actually suppress them), and could have
an incentive of vaccination based on their social interactions. Usually,  macroscopic
indicators at the level of city or state give a clue of the severity of the epidemic, although  as
in the case of the actual Covid-19 pandemic, an exact knowledge of the state of your
contacts is challenging due to the presence of asymptomatic individuals.

\bigskip

Then, we study the optimal controls for the vaccination formulated as a game both in a
decentralized and centralized setting following ideas  developed in \cite{doncel2017mean}
(in a purely mean-field setting). We show in particular that the optimal vaccination strategy
boils down to a bang-bang control, i.e., the optimal solution consists in vaccinating with the
highest possible rate until some fixed time-threshold depending on the connectivity of the
network and the costs, and then not to vaccinate at all anymore. On the other hand, using
techniques from continuous optimization for systems with restricted controls, we first define
a general assumption that allows to prove uniqueness and existence of optimal control in the
sense of viscosity solutions following \cite{bresan} and \cite{trelat2008controle}. We also
 show that the optimal centralized strategy must be developed at the highest rate possible when deployed. We could not however prove that
 there is only one phase of vaccination.

Finally, we consider four network examples, compute the optimal vaccination strategy and
simulate the propagation of the disease and the vaccination program. We observe that the
optimal strategy has the desirable feature to reduce considerably the total number of
infected, while a more conservative vaccination program would increase costs without having
significant effects on the epidemics. We also relate our results to graph measures as
centrality coefficients and to the classical indicators associated to the epidemic.

\subsection{Organization of the paper and main contributions}

In Section \S 2 we introduce the necessary notation and  the epidemic model.  For a sake of
completeness, we add a short description of the configuration model random graph, together
with the relevant measure spaces considered.

\bigskip

In Section \S 3 we generalize the results of \cite{decreusefond2012large} describing the
propagation of an epidemics on a configuration model random graph by incorporating a
strategy of vaccination for the susceptible population depending arbitrarily on time, and depending linearly on the node
degree.

We present our main fluid limit result, we obtain an infinite system of measure valued
differential equations that describe the Markovian dynamics of propagation of a disease on
the random graph, subject to a vaccination strategy, see Theorem \ref{mainTh}. Also,
 we derive a finite-dimensional differential system
describing the evolution of the main variables that describe the epidemics, namely the
number of individuals in each compartment ($S$, $I$, $R$, $V$), and the probability of
interaction between the susceptible population with agents in different compartments. In
particular, $p^X$ will denote the probability that an edge connects a susceptible with a node
in state $X$, for $X=S$, $I$, $R$, or $V$. We use a generalization of the probability generating function
$g$, involving both the initial degree of agents and the probability that a degree one node remains susceptible, described by $\alpha$ and $\theta$. The closed system we get is the following (see the next section for the precise
definitions and notation):
\begin{equation}\label{MyEq}
\left\lbrace
\begin{aligned}
&\dot{\alpha}=-rp^I\alpha\\
&\dot{\theta}=-\pi\theta \\
&\dot{I}=-\gamma I + rp^I\alpha \partial_\alpha g(\alpha,\theta)\\
&\dot{V}=\pi\theta \partial_{\theta} g(\alpha,\theta)\\
&\dot{p^S}=rp^Ip^S\left(1-\frac{\alpha \partial_{\alpha\alpha} g(\alpha,\theta)}{\partial_{\alpha} g(\alpha,\theta)}\right)-\pi p^S-\theta \pi p^S \frac{\partial_{\alpha\theta} g(\alpha,\theta)}{\partial_{\alpha} g(\alpha,\theta)}\\
&\dot{p^I}=-\gamma p^I+rp^Ip^S\frac{\alpha \partial_{\alpha\alpha} g(\alpha,\theta)}{\partial_{\alpha} g(\alpha,\theta)}-rp^I(1-p^I) \\
&\dot{p^V}=rp^Ip^V+\theta\pi p^S \frac{\partial_{\alpha\theta} g(\alpha,\theta)}{\partial_{\alpha} g(\alpha,\theta)}\\
\end{aligned}
\right.
\end{equation}
where $r$, $\gamma$ and $\pi_t$ are the contagion, recovery and vaccination rates.

The heterogeneity of the population connections, which is an important feature in many
epidemic propagation models
\cite{moreno2002epidemic,RevModPhys,anderson1992infectious,turnes2014epidemic}, can
be grasped here through the function $g$ of the degree distribution; in our case,
through the expression $\frac{\alpha_t \gxx{t}}{\gx{t}}$.

Let us observe that this heterogeneity is not present in classical mean field models which
rely on the assumption of a completely homogeneous mixed population. However, we will show that the system \eqref{MyEq}  in a configuration model
with a Poisson degree distribution converges to the classical SIR model when the mean
degree goes to infinity.
The system of
equations \eqref{MyEq} allows us to compare agent based simulations or real data with the
curves obtained by numerical integration \cite{foguelman2020ebdevs}. We also derive a simpler system for the case of degree-proportional and constant
vaccination.

\bigskip

In Section \S 4, we study the optimal control problems associated with vaccination.  In order to
consider the effect of the graph structure, we define  a vaccination strategy as a bounded
and measurable time dependent function, followed uniformly by all the population but
depending on the connectivity of an individual. We assume that the rate of vaccination is not decreasing in the degree, in agreement with the
literature where highly connected individuals have more incentive to get
vaccinated whereas individuals with few contacts have less incentive. On the other hand,
this is already a large family of controls, needing a quite general theoretical treatment,
involving in particular weak viscosity solutions.

The optimal vaccination policy will vary if we consider the individual point of view;
this decentralized case is approached using the theory of mean field games, by considering
the perspective of a single rational individual added to an infinite population with a given
arbitrary vaccination policy. We show in particular that the optimal vaccination strategy boils
down in that case to a bang-bang control, i.e., the optimal solution consists in vaccinating with the
highest possible rate until some fixed stopping time depending on the connectivity of the
network and the costs, and then not to vaccinate at all anymore,  $$\pi_t=\nu\mathds{1}_{[0,\tau]}(t)$$
where $\nu$ is the rate of vaccination and
$\mathds{1}_{[0,\tau]}$ is the characteristic function of the set $[0,\tau]$. Then, we
consider the social optimum, and  we introduce a particular cost functional in order to find
the optimal centralized strategy. As before, we show that the optimal strategy $\pi_t$ is of threshold
type, in the sense that takes the maximum value or zero.

\bigskip

In Section \S 5 we analyze the theoretical results and we present numerical computations.
We obtain a modified Basic Reproduction Number for the epidemics on the graph before the
vaccination process started, namely
\[R_0=\frac{r}{r + \gamma}\frac{\partial_{\alpha\alpha}g(1,1)}{\partial_{\alpha}g(1,1)}.\label{R0}\]
Therefore, an epidemic outbreak  will occur if $R_0>1$, and corresponds to the already
known critical threshold stated in  \cite{RevModPhys,PhysRevE.66.016128}. We compute the
optimal vaccination strategy and we simulate both the propagation of the disease and the
vaccination program. We observe that the optimal strategy has the desirable feature to
reduce considerably the final size of the epidemic, depending on the contagion and recovery
rates $r$ and $\gamma$,  the function $g$, the threshold $\tau$ and the rate of vaccination
$\nu$, while a more conservative vaccination program would increase costs without having
significant effects on the epidemics.

Finally, we consider different networks generated by four degree distributions, all with the same
mean degree first, and with the same $R_0$ later:

\begin{itemize}
\item[(a)]
Poisson, which corresponds with mean field models with homogeneous mixing,
as we show in \ref{meanfield};

\item[(b)] Bimodal, where a fraction of the population has few links
distributed Poisson, while the rest has many links, and is very interesting and important
question now, related to the Covid-19 pandemics, and it is the rationale behind school
closure;

\item[(c)] a Regular graph, as an example of networks where all nodes are equals on its rates;
and

\item[(d)] Power Law, a classical model of social networks
\cite{newman2003structure}.
\end{itemize}

We conclude in Section \S 6, and the full proofs of the different theorems can be found in
Section \S 7.

\section{Model Setting and notation}

\paragraph{Configuration Model Graph}
Let us introduce the stochastic environment of the epidemic process. We use the
Configuration Model random graph introduced by Bollob\'as \cite{bollobas1998random}
which can be constructed as follows. First denote $\mathbb{N}_0^n=\{0,...,n\}$ and suppose
we have $n$ nodes, and a  sequence of degrees $k_1,\ldots, k_n$ independent and
identically distributed according to $p^{(n)}=(p^{(n)}_k)_{k=1,...,n}$ \footnote{The
independence assumption can be relaxed but we assume it for simplicity.} such that the sum
of the degrees is even. One initially assigns a quantity $k_i$ of half-edges to the $i$th-node
and then choose two of them uniformly from the unmatched ones, establishing the
connection between the nodes, until all the half-edges are matched. Nodes will represent
individuals, thus we will use both terminology throughout the paper.

Under the assumption that $p^{(n)}$
converges in probability to $p=(p_k)_{k\in \mathbb{N}_0}$ and $E[{p^{(n)}}^2]$ converges to
$E[p^2]<\infty$ when $n$ goes to infinity, there is asymptotically a probability bounded away
from 0 of obtaining a simple graph as showed in \cite{janson2009}. Thus we may repeat the
matching procedure until we obtain a simple graph, i.e., until the resultant graph does not
contain self-loops nor multiedges \cite{durrett2007random}.

As a consequence, the degree of a randomly chosen node is distributed according to
$p^n$. Now, the probability that a neighbor has degree $k$ is $\frac{k
p_k}{\sum_{j=1,...,n}jp_j}$, which is the so-called size-biased degree distribution. This distribution
will greatly influence the dynamics on configuration models as opposed to the mean
field point of view, where the underlying graph of potential connections is complete and both the degree distribution
and the size-biased distribution of the graph induced by the contact process coincide (and are asymptotically Poisson).

Given a degree distribution $p=(p_k)_{k\in\N}$, the associated probability generating function is defined by $\psi(z)=\sum_{k\in\N}p_k z^k$.

\paragraph{Epidemic dynamics and vaccination}

We now describe the dynamics of the epidemics with vaccination. For a given Susceptible
node (i.e. not having contracted the illness nor vaccinated) we consider several independent exponential
clocks with parameter $r$, one for each edge connecting an Infected node (i.e., potential
encounters). It will describe the contact process: if this clock rings, the Susceptible makes a
transition to state Infected and remain infectious during an exponential time with mean
$1/\gamma$, whereupon it will not longer infect any node, going to Recovered state. Also for
the Susceptible, we consider another exponential clock with parameter $\pi_t(k)$ depending
on the degree $k$ of the node, which is the time dependent control variable and represents
the rate at which the individual becomes Vaccinated.

Both Recovered and Vaccinated are absorbing states of the resulting continuous time Markov chain, and we differentiate between both to keep track of the epidemics characteristics. Also, they have a different impact on the total costs of an epidemic.

In the sequel, we suppose that $\pi_t(k)=\xi(k) \pi_t$, where $\pi_t$  is a measurable bounded function on
$[0,T]$ and $\xi:\N \to \R$. For the existence of the fluid limit, some technical
assumption on $\xi$ is needed, and we suppose that
$$ \sum_k\mu^S_t(k)\xi(k)< \pint{\pi_t\mu^S_0}{\chi^3}.   $$ 
where $\mu^S$ is the degree distribution of the susceptible population.

The cases $\xi(k)=1$ and $\xi(k)=k$ appear more frequently in the literature of local
vaccination strategies. Constant rate indicate homogeneity of the agents on the
immunization effort; and a rate proportional to the degree, as in the case of acquaintance vaccination
\cite{ball2013acquaintance,britton2007graphs,cohen2003efficient}, takes into account that vaccination of hubs and higher degree nodes is convenient to stop the transmission of a disease. This can also be achieved using a targeted vaccination where higher degree nodes are vaccinated, which corresponds to
$\xi(k)=\mathds{1}_{\kappa \leq k}(k)$ for some threshold $\kappa$. In the acquaintance vaccination, one samples a fraction of the nodes, and for each sampled node, one of their
neighbors uniformly chosen is vaccinated. Hence, the probability of a node to be selected by any
neighbor is proportional to their degree since it has the size biased distribution.

\paragraph{SIR-V dynamics on a Configuration Model}
We denote $S_t^n$, $I_t^n$, $R_t^n$ and $V_t^n$ the total number of Susceptible, Infected,
Recovered and Vaccinated nodes respectively, and $\mathcal{S}_t^n$, $\mathcal{I}_t^n$, $\mathcal{R}_t^n$ and $\mathcal{V}_t^n$ the sets containing the nodes in each state. These quantities are of central interest in most of the
epidemics literature and they are the main variables on which equations of the dynamics are
exhibited (we refer the reader to \cite{kiss2017mathematics} for an informative review on
epidemics dynamics). Nevertheless, let us remark that our model being over a Configuration
Model random graph, computing its dynamics is in principle very demanding, as we should
study a stochastic process in (growing) dimension $n$, the number of nodes.

The principal
signification of part (ii) in Theorem \ref{mainTh} below is that the limit behavior is a good approximation
of the case $n$ large.

Instead, we study the dynamics of four measures in $\M$, the set of finite
measures on $\N_0$ embedded with the topology of weak convergence (we refer the reader
to \cite{billingsley2013convergence} for a complete description of topological properties and
results), describing the connection between the susceptible population and the rest. For this
purpose, we resort to the so-called \textit{principle of deferred decisions}, revealing the
graph simultaneously with the propagation of the disease, regarding the types of the edges
connecting the different states of the individuals. This trick is possible since the random
environment for the epidemics  dynamics acts over a Configuration Model which is
constructed using a uniform matching. We follow here the ideas developed in
\cite{decreusefond2012large} and \cite{mattmoyalbermolen}.

We now describe the quantities involved in our formulation of the dynamics precisely. For the $i$th-node, we denote $k_i^S$ the random number of edges connecting the $i$th-node to a
susceptible individual, and $\delta_k$ the Dirac measure on $\N_0$. The empirical
measure $\mu_t^{S,n}\in \M$  describes the degree of the susceptible individuals: for each
$k\in\N_0$, $\mu_t^{S,n}(k)$ denotes the number of susceptible nodes with degree $k$ at
time $t$,
 \[\mu_t^{S,n}=\sum_{i\in \mathcal{S}_t^n}\delta_{k_i}.\]
 Similarly, $$
\mu_t^{IS,n}=\sum_{i\in \mathcal{I}_t^n}\delta_{k^S_i} ,\quad \mu_t^{RS,n}=\sum_{i\in \mathcal{R}_t^n}\delta_{k^S_i} \; \text{ and } \; \mu_t^{VS,n}=\sum_{i\in \mathcal{V}_t^n}\delta_{k^S_i}
$$ represent the number of nodes in each state $I$, $R$, $V$ connected with the susceptible population. We write
 $\mu_t^n=(\mu_t^{S,n},\mu_t^{IS,n},\mu_t^{RS,n},\mu_t^{VS,n})$.

Instead of considering the proportion of the population in each state,
our work describes the dynamics of the following
edge-based quantities:
\[ N^{S,n}_t= \langle \mu^{S,n}_t,\chi \rangle := \sum_{k\in \mathbb{N}}k\mu^{S,n}_t(k), \] the number of semi-edges connecting a susceptible node;
and, analogously, $N^{IS,n}_t$, $N^{RS,n}_t$, $N^{VS,n}_t$ are the number of edges linking a susceptible node
with an infected, recovered or vaccinated one,  when the size of the population
is $n$.

We also consider the proportions of edges associated to those quantities:
\begin{align*}
p^{I,n}_{t}=& \frac{N^{IS,n}_t}{N^{S,n}_t}, \\
p^{R,n}_{t}=&\frac{N^{RS,n}_t}{N^{S,n}_t}, \\
p^{V,n}_{t}=&\frac{N^{VS,n}_t}{N^{S,n}_t} \\
p^{S,n}_{t}=&\frac{N^{S,n}_t-N^{IS,n}_t-N^{RS,n}_t-N^{VS,n}_t}{N^{S,n}_t}.\end{align*}

\paragraph{Scaling limits} Our main result consists in the convergence of the normalized empirical
measures in the Skorokhod space embedded with the weak topology. We write the semimartingale
decomposition stated in Proposition 5 and a description of the Markovian process through a system
of stochastic differential equations derived form Poisson point measures, see section 7.1.

\newpage

For each  $\mu\in\M$, the set of finite measures on the natural numbers including zero, and $f\in \B_b(\N_0)$, the set of bounded real
functions on $\N_0$, we write $$\pint{\mu}{f}=\sum_{k\in \N_0} f(k)\mu(k).$$

In order to show that the normalized degree empirical measures of each type converge to the solutions
of an
infinite system of differential equations as the population size tends to infinity, we scale the
measures in the following way: for $n\in\N$, we set $$\mu^{(n)}_t=\frac{1}{n}\mu^n_t.$$

The suppression of $n$ will denote the limit measures associated to the fluid limit $\mu_t=(\mu_t^{S},\mu_t^{IS},\mu_t^{RS},\mu_t^{VS})$. In the limit we also define:
\begin{equation}
\alpha_t=e^{-\int_0^t r p^{I}_{s}ds} \text{ and }
\theta_t=e^{-\int_0^t \pi_s ds}
\end{equation}
that allow us to describe the probability that a node with degree $k$ has not contracted the disease, $\alpha_t^k$, and the probability of no vaccination at time $t$, $\theta_t^{\xi(k)}$.

\section{Results}

\subsection{Large graph limit}

We now state the convergence of the degree empirical measures when the size of the
population tends to infinity. The
corresponding deterministic solution of the measure-valued system will in turn give
interesting insights on the effect of the vaccination in the propagation of the epidemic for
large populations.

The proof is similar to the proof of the main theorem in \cite{decreusefond2012large}, for a SIR without vaccination.
We use strongly the fact that $\pi_t$ is uniformly bounded on time,
 and this introduces several differences in steps 2 and 5 (co,pared to their original proof), where we need to invoke results from
 \cite{dipernalions} in order to ensure uniqueness of weak solutions to the transport equation involved.
 We add it in the last section for a matter of completeness.

\begin{theorem}
\label{mainTh} Suppose $(\mu^{(n)}_0)_{n\in\N}$ converges to $\mu_{0}$ in
$\mathcal{M}_F^4(\mathbb{N}_0)$ embedded with the weak topology. Then \begin{enumerate}
\item[(i)] there exists a unique solution $\mu_{t}$ of the deterministic system of
    equations (\ref{GenEq});
\item[(ii)] the sequence $(\mu^{(n)})_{n\in\N}$ converges in
    distribution to $\mu$ in the Skorokhod space when $n$ goes to infinity.
\end{enumerate}
\begin{equation}\label{GenEq}
\left\lbrace
\begin{aligned}
\langle \mus{t},f\rangle\; = &\sum_{k\in \mathbb{N}}\mus{0}(k)\alpha_t^k \theta_t^{\xi(k)} f(k), \\
\langle \muis{t},f\rangle = &\langle \muis{0},f\rangle - \int_0^t \gamma \langle \muis{s},f\rangle ds + \int_0^t \sum_{k\in \mathbb{N}} r\p{I}{s}k  \\
 &\quad \times \sum_{\mathclap{\substack{i,j,l,m / \\ i+j+l+m=k-1}}}  \quad {k-1 \choose i,j,l,m} (\p{S}{s})^i(\p{I}{s})^j(\p{R}{s})^l(\p{V}{s})^m \mus{s}(k) f(i) ds  \\
 &+\int_0^t \sum_{k\in \mathbb{N}}rk\p{I}{s}(1+(k-1)\p{I}{s})\sum_{j\in \mathbb{N}_0}(f(j-1)-f(j))\frac{j\muis{s}(j)}{N_s^{IS}}\mus{s}(k)ds \\
&+\int_0^t \sum_{k\in \mathbb{N}} \pi_s(k)k\p{I}{s}\sum_{j\in \mathbb{N}_0}(f(j-1)-f(j))\frac{j\muis{s}(j)}{N_s^{IS}}\mus{s}(k) ds. \\
\langle \murs{t},f\rangle = &\langle \murs{0},f\rangle + \int_0^t \gamma \langle \muis{s},f\rangle ds + \int_0^t \sum_{k\in \mathbb{N}} (rk\p{I}{s}(k-1)\p{R}{s}+\pi_s(k)\p{R}{s}k)  \\
&\qquad\times\sum_{j\in \mathbb{N}_0}(f(j-1)-f(j))\frac{j\murs{s}(j)}{N_s^{RS}}\mus{s}(k) ds. \\
\langle \muvs{t},f\rangle = &\langle \muvs{0},f\rangle \\
&+ \int_0^t \sum_{k\in \mathbb{N}} \pi_s(k) \quad \sum_{\mathclap{i+j+l+m=k}} \quad \; {k \choose i,j,l,m} (\p{S}{s})^i(\p{I}{s})^j(\p{R}{s})^l(\p{V}{s})^m \mus{s}(k) f(i) ds  \\
&+ \int_0^t \sum_{k\in \mathbb{N}} (r\p{I}{s}k(k-1)\p{V}{s}+\pi_s(k)\p{V}{s}k)\;\mus{s}(k) \\ &\qquad \times \sum_{j\in \mathbb{N}_0}(f(j-1)-f(j))\frac{j\muvs{s}(j)}{N_s^{VS}} ds.
\end{aligned}\right.
\end{equation}
\end{theorem}

The derivation of the system of equations (\ref{GenEq}) can be explained by
the underlying Markovian process describing the epidemics on the configuration model, which
is also useful for a possible simulation. The edges based measures must be updated when
any of the three possible events occur: the infection or vaccination of a susceptible individual
(with a rate that depends linearly on their degree $k$), and the removal of an infected one
(according to exponential clocks of parameter $\gamma$).

The probability that an individual of degree $k$ that started susceptible remains in that state
at time $t$ is $\theta_t^{\xi(k)}\alpha_t^k$, which gives the first equation. In the second equation,
 the first integral accounts for the removing of infectious individuals; the second
one corresponds to the addition of the newly infected, $rkp^I$ is the rate of infection for a
susceptible of degree $k$ and the multinomial part is the random draw of the neighbors
according to the probabilities of the edges. The third and fourth integrals correspond to the
infection or vaccination of the neighbors of each infected individual, whose degree is chosen
according $\frac{j\muis{s}(j)}{N_s^{IS}}$, the size biased distribution. The number of edges to
the susceptible population decreases by one in both events, infection and vaccination.

Now, choosing $f(k)=\mathds{1}_i(k)$ in (\ref{GenEq}) we obtain a countable system of ordinary
differential equations that allows us to describe the infection propagation in terms of the
measures:

$$
\left\lbrace
\begin{aligned}
\mus{t}(i)\;= &\mus{0}(i)\alpha_t^i\theta_t^{\xi(i)}  \\
\muis{t}(i) = &\muis{0}(i) - \int_0^t \gamma \muis{s}(i) ds + \int_0^t r\p{I}{s}  \\
 &\qquad \times \sum_{j,l,m} (i+j+l+m+1) {i+j+l+m+1 \choose i,j,l,m} \\
 & \qquad \times (\p{S}{s})^i(\p{I}{s})^j(\p{R}{s})^l(\p{V}{s})^m \mus{s}(i+j+l+m+1) ds  \\
 &+\int_0^t \left( r\p{I}{s}\langle\mus{s},\chi\rangle +r(\p{I}{s})^2\langle\mus{s},\chi^2-\chi\rangle+\p{I}{s}\langle\mus{s}\pi_s,\chi \rangle\right) \\
 & \qquad \times \left(\frac{(i+1)\muis{s}(i+1)-i\muis{s}(i)}{\langle\muis{s},\chi\rangle} \right)ds \\
\murs{t}(i) = &\murs{0}(i) + \int_0^t \gamma \muis{s}(i) ds + \int_0^t \left( r\p{I}{s}\p{R}{s}\langle\mus{s},\chi^2-\chi\rangle+\p{R}{s}\langle\mus{s}\pi_s,\chi \rangle\right)  \\
&\qquad \times\left(\frac{(i+1)\murs{s}(i+1)-i\murs{s}(i)}{\langle\murs{s},\chi\rangle} \right)ds\\
\muvs{t}(i) = &\muvs{0}(i)+\int_0^t \sum_{j,l,m} \pi_s(i+j+l+m) {i+j+l+m \choose i,j,l,m} \\
& \qquad \times (\p{S}{s})^i(\p{I}{s})^j(\p{R}{s})^l(\p{V}{s})^m \mus{s}(i+j+l+m) ds \\
&+ \int_0^t \left( r\p{I}{s}\p{V}{s}\langle\mus{s},\chi^2-\chi\rangle+\p{V}{s}\langle\mus{s}\pi_s,\chi \rangle\right) \\
&\qquad \times\left(\frac{(i+1)\muvs{s}(i+1)-i\muvs{s}(i)}{\langle\muvs{s},\chi\rangle} \right)ds.\\
\end{aligned}\right.
$$

\subsection{Closed system}\label{closed}

Now, we derive a generalization of the equations proposed by Volz as we show in Proposition
\ref{flips} including a generic vaccination strategy.

If the vaccination function we have added to the model were constant or bounded and continuous,
then there would be nothing to do in order to have an existence and uniqueness result. However, we
are allowing the control to be bounded and measurable implying the necessity of a general
treatment as the studied in \cite{bresan} or \cite{trelat2008controle} which states that the
functional describing the dynamics needs to be Lipschitz. This hypothesis is
straightforward using classical computations and bounds together with the properties of the involved
functions.

We define the function
\begin{equation}\label{g}
g(\alpha,\theta)=\sumn{k} \mu_0^S(k)\alpha^k \theta^{\xi(k)}
\end{equation}
and denote with a subscript the partial derivation of g, namely:
\[\partial_\alpha g(\alpha,\theta)=\sumn{k} \mu_0^S(k)\alpha^{k-1} k \theta^{\xi(k)}\] and analogously for the rest. This function will be key to reduce the number of
equations to seven, obtaining in this way a tractable description of both the
epidemics and the optimal vaccination strategy.

\begin{proposition}\label{flips}
In the limit of infinitely many nodes, system \eqref{GenEq} can be reduced  to the following
system of differential equations,
\begin{equation}\label{MyEqZ}
\left\lbrace
\begin{aligned}
&\dot{\alpha}=-rp^I\alpha\\
&\dot{\theta}=-\pi\theta \\
&\dot{I}=-\gamma I + rp^I\alpha \partial_\alpha g(\alpha,\theta)\\
&\dot{V}=\pi\theta \partial_{\theta} g(\alpha,\theta)\\
&\dot{p^S}=rp^Ip^S\left(1-\frac{\alpha \partial_{\alpha\alpha} g(\alpha,\theta)}{\partial_{\alpha} g(\alpha,\theta)}\right)-\pi p^S+\theta\pi p^S\frac{\partial_{\alpha\theta} g(\alpha,\theta)}{\partial_{\alpha} g(\alpha,\theta)}\\
&\dot{p^I}=-\gamma p^I+rp^Ip^S\frac{\alpha \partial_{\alpha\alpha} g(\alpha,\theta)}{\partial_{\alpha} g(\alpha,\theta)}-rp^I(1-p^I) \\
&\dot{p^V}=rp^Ip^V+\theta\pi p^S \frac{\partial_{\alpha\theta} g(\alpha,\theta)}{\partial_{\alpha} g(\alpha,\theta)}\\
\end{aligned}
\right.
\end{equation}
Moreover, the right-hand side of the system is Lipschitz and uniformly bounded, and
the problem (\ref{MyEqZ}) hence admits a unique solution for any initial datum and any
measurable $\pi:[0,T]\to[0,\nu]$.
\end{proposition}
\begin{proof}
We use the function $g$ to compute a closed expression for $N^S_t$, $N^{IS}_t$, $N^{RS}_t$ and $N^{VS}_t$ and its derivatives. We denote
$\mathds{1}(k):=1$ and $\chi(k)=k$.

 Note that \[S_t=\langle \mus{t},\mathds{1} \rangle = \sum_{k\in \mathbb{N}}\mus{t}(k) = \sum_{k\in
 \mathbb{N}}\mus{0}(k)\alpha_t^k \theta_t^{\xi(k)} = g(\alpha_t,\theta_t)\] is the proportion of susceptible individuals at time $t$.

In a similar way, \begin{align*} I_t= & \langle \muis{t},\mathds{1} \rangle \\ = & \sum_{k \in
\mathbb{N}} \muis{0} - \int_0^t \gamma I_s ds + \int_0^t r\p{I}{s}\alpha_s \partial_{\alpha} g(\alpha_s,\theta_s) ds \\
=
& I_0 + \int_0^t -\gamma I_s + r\p{I}{s}\alpha_s \partial_{\alpha} g(\alpha_s,\theta_s) ds,\\
R_t= & R_0 + \int_0^t \gamma I_s ds,\\
V_t= & V_0 + \int_0^t \pi_s \theta_s \alpha_s \partial_{\theta} g(\alpha_s,\theta_s)ds.
\end{align*}

The next step is to find the dynamics for $\p{S}{t}$ and the other edges probabilities. Before
computing it, let us note that:
\begin{equation}\label{NS}
N^S_t=\langle \mus{t},\chi \rangle =\sumn{k}\mu_t^S(k)k= \sum_{k\in \mathbb{N}}\mus{0}(k)\alpha_t^k k \theta^{\xi(k)}=  \alpha_t \partial_{\alpha} g(\alpha_t,\theta_t).
\end{equation}

Using the definition of $\p{I}{t}$ and $\dot{\alpha}_t=-r\p{I}{t}\alpha_t$, we obtain
\[\frac{\dot{N}^S_t}{N^S_t}=-r\p{I}{t}-\frac{\theta_t \pi_t\gxy{t}}{\alpha_t\gx{t}}-\frac{\alpha_t rp^I_t \gxx{t}}{\gx{t}}.\]
We replace $f$ by $\chi$ in (\ref{GenEq}), and after some basic computations, by rearranging
 terms using the multinomial theorem, we have
\begin{equation}
\begin{aligned}
\dot{\p{I}{t}}=&\frac{\dot{N^{IS}_t}}{N^S_t}-p^I_t\frac{\dot{N^{S}_t}}{N^S_t}=\\
=&-\gamma\p{I}{t}+\p{I}{t}r\p{S}{t}\frac{\alpha_t\gxx{t}}{\gx{t}}-r\p{I}{t}(1-\p{I}{t})
\end{aligned}
\end{equation}

Reasoning similarly with the other probabilities, and putting all the equations
together, we have, for each control $\pi:[0,T]\to[0,\nu]$, the closed system of equations
\eqref{MyEqZ}.

This finishes the proof.
\end{proof}

\begin{remark}
The particular case $\xi(k)=ak+b$ which includes constant and acquaintance vaccination can
be studied using only the generating function of the initial degree
$\psi(z)=\sum_{k\in\N}\mu_0^S(k) z^k$ of the network. To this end, we define the quantities
\[
\beta_t=e^{-\int_0^t r p^{I}_{s}+a\pi_sds} \text{ and }
\phi_t=e^{-\int_0^t b\pi_s ds}
\]
and rewrite the system in terms of these new variables. Let us note that the derivation follows
from system \eqref{GenEq}.
 Noticing that $S_t=\phi_t \psi(\beta_t)$, the dynamics are described by:
\begin{equation}\label{eqConst}
\left\lbrace
\begin{aligned}
&\dot{\beta}=(-rp^I-a\pi)\beta\\
&\dot{\phi}=-b\pi\phi \\
&\dot{I}=-\gamma I + rp^I \phi \beta \psi'(\beta)\\
&\dot{V}=a\pi\phi\beta \psi'(\beta)+b\pi\phi \psi(\beta) \\
&\dot{p^S}=\left(1-\frac{\beta \psi''(\beta)}{\psi'(\beta)}\right) p^Srp^I-\pi p^S(a+b)-p^S\pi a \frac{\beta \psi''(\beta)}{\psi'(\beta)}\\
&\dot{p^I}=-\gamma p^I+p^Irp^S\frac{\beta \psi''(\beta)}{\psi'(\beta)}-rp^I(1-p^I) \\
&\dot{p^V}=p^S\pi(a+b)+rp^Ip^V+\pi p^S a \frac{\beta \psi''(\beta)}{\psi'(\beta)}.
\end{aligned}
\right.
\end{equation}
Observe first that the dynamics of the epidemics depend strongly on the degree distribution,
since the expression $\frac{\beta_t \psi''(\beta_t)}{\psi'(\beta_t)}$ governs the expected
number of susceptible individuals connected with a neighbor of a given degree at time $t$.
\end{remark}

\subsubsection{Relation with mean field models}\label{meanfield}

Let us quickly clarify the differences between this model with vaccination rate linear on the degree and the usual mean field
model. In the sparse Erd\"{o}s-Renyi model, the number of neighbors in the graph follows a
binomial distribution, which can be approximated in a large population by a Poisson
distribution. On the other hand, when the graph is fully connected and the contact process is
determined by a Poisson process, the number of neighbors with whom each node effectively
connects is also Poisson distributed.

In the Mean Field model with vaccination \cite{doncel2017mean}, an individual of an
homogeneous population encounters others following a Markov process in continuous time
with rate $r$. So, = individuals can be in three states:
susceptible, infected and recovered or vaccinated; and we denote $S_t$, $I_t$ and $R_t$, its
respective proportions of the total population. Here, vaccinated individuals are treated as
recovered, since the influence in the propagation of the disease is the same. If the initial
individual of the contact process is susceptible and the encountered one is infected, the first one
becomes infected. An infected individual recovers at rate $\gamma$, and a susceptible can choose
its own vaccination rate $\pi$, going to recovery state.
The optimal strategy $\pi$ played for all the players is called a mean-field equilibrium, defined as a fixed point of the
best response functional, this is, $\pi \in BR(\pi)$, which minimize a properly defined cost functional \cite{doncel2016mean}.
A mean ﬁeld equilibrium
consists in a strategy where no player has an incentive to deviate from the common strategy for
their own benefit.
When the size of the population goes
to infinity, the dynamics of the population where all players use the vaccination strategy
$\pi$ is described by the following system of equations:

\begin{equation}\label{MForiginal}
\left\lbrace
\begin{array}{l}
\dot{S}=-rIS-\pi S\\
\dot{I}=rIS-\gamma I\\
\dot{R}=\gamma I+\pi S
\end{array}
\right.
\end{equation}

The main difference between the two systems lies on the term associated to the infection
process in the mean field case, $rS_tI_t$, which is now $r\p{I}{t}\phi_t \psi'(\beta_t)=rN^{IS}_t$ and
the neighbors are chosen according to the size biased distribution. In the particular case of
Poisson distribution, the size biased distribution is also Poisson, and, as we show in this
section, both models are asymptotically similar.

Here, we are considering 
an edge-based dynamics instead of an individual-based one. In the propagation of
the disease, exponential clocks of the contact process are assigned to edges connecting
susceptible with infected nodes, the mechanism is not to choose uniformly between all the
population but in a neighborhood, which modifies quantitative and qualitatively the generator of the Markov
process and therefore the limit equation.

Though the dynamics are not the same, (even in the case of a Poisson distribution for the
configuration model), one can actually show that they are asymptotically equivalent, when the
mean number of connections between individuals grows large.

If we consider a population of size $N$ in which every individual has $C$ possible contacts
and scale it such that $\hat{r}=rC$ remains constant, the mean field equation of these
dynamics is described by
(\ref{MForiginal}) replacing $r$ and $\pi$ by $\hat{r}$ and $\hat{\pi}$ respectively. Then our limiting dynamics on a configuration model with Poisson degree distribution, not
fully connected but uniformly linked, solve the MF equations when $C$ goes to infinity taking
$\psi(z)=C e^{C(z-1)}$. Being $S=\phi \psi(\beta)$, we have:
\[\dot{S}=\phi \beta \psi'(\beta) (-rp^I+a\pi) - \pi b S_t=
-CSrp^I\beta-a \beta \pi C S-b\pi S \]

Since $\hat{r}$ is taken to be constant, it is of order $O(1)$ as $C$ grows, and only a
proportion of order $O(I/C)$ of the edges may transmit the infection from one infected
neighbor to the observed susceptible. Also, the last term is of order $O(1)$. Therefore, for large $C$, $p^I$ can be approximated by
$I-O(I/C)$ and similarly, when $C$ is large enough, $\beta = 1-O(1/C)$. Moreover,
$p^I\beta=I+O(I/C)$, giving us \[\dot{S}=-\hat{r}IS-\hat{\pi}S+O(IS/C),\] which is
asymptotically the first equation in (\ref{MForiginal}). The third equation follows from a similar
reasoning, and the second one from the fact that $S=1-I-R$.

\section{Optimal Control Problems}

As underlined before, we aim in this Section at describing the optimal strategy which will
consist in vaccinating at the maximum rate as soon as possible up to some critical time (not
necessarily deterministic) depending on the parameters of the model.

We define a generic vaccination strategy, which is time dependent and also depends on the
individual connection in the network environment. Using the theory developed in
\cite{bresan,trelat2008controle}, we characterize under mild assumptions the
existence and uniqueness of a viscosity solution for our optimization problem. Of course, the
model is a crude simplification of a very complex reality but we believe it still allows to grasp
an interesting phenomenology (see Section \ref{lasecdesimu}). We do not involve explicitly in
our setting more individual features like age, risk, or beliefs, nor susceptibility to the disease.
However, many of these characteristics may be modeled through adequate parameters   and  cost values.

\subsection{Individual Cost}

In this section, we analyze the optimal control problem from an individual point of view. We
focus on the perspective of a particular individual immersed in the population, who will take
decisions in order to minimize their cost in a game against the whole population. This rational
individual can be considered as a player seeking for the best response to a fixed strategy
followed by the population, and we are therefore in the context of the theory of mean field
games, which shall provide our context, definitions of equilibrium and existence results.

Suppose we add an individual to a population that evolves according to a vaccination
strategy $\pi$. Since the population is infinite, the behavior of this new individual will not
affect the evolution of the whole population, hence its dynamics will be described by the
already stated equations (\ref{MyEqZ}), which we represent in the form $\dot{x}=\varphi(x,\pi)$. We denote by
$\tilde{\pi}$ the vaccination strategy for the new individual and
$\tilde{x}=(\tilde{S_t},\tilde{I_t},\tilde{R_t},\tilde{V_t})$ their probability distribution over the four
possible states.

Finally, let us remark that despite the mean field game denomination, the
dynamics of the population follows the network based system that we described in section
\ref{closed}.

We suppose that the new individual has degree $k$, therefore, we have that $$\tilde{ \theta}_t=e^{-\int_0^t \tilde{\pi}_s ds}.$$ Since $\tilde{S_t}=\alpha_t^k\tilde{\theta}^{\xi(k)}$ as before, but depending on
the dynamics of the population infected edges and on their own vaccination rate, their
state is determined by the following system of the form
$\dot{\tilde{x}}=f_0(x,\tilde{x},\pi,\tilde{\pi})$:
\begin{equation}
\left\lbrace
\begin{aligned}
&\dot{\tilde{\theta}}=-\tilde{\pi}\tilde{\theta}\\
&\dot{\tilde{I}}=-\gamma \tilde{I} + rp^I\alpha g'(\alpha,\tilde{\theta})\\
&\dot{\tilde{V}}=\tilde{\pi}\tilde{\theta} \partial_\alpha g(\alpha,\tilde{\theta})\\
&\dot{\alpha}=-rp^I\alpha \\
&\dot{\theta}=-\pi\theta \\
&\dot{I}=-\gamma I + rp^I\alpha \partial_\alpha g(\alpha,\theta)\\
&\dot{V}=\pi\theta \partial_{\theta} g(\alpha,\theta)\\
&\dot{p^S}=rp^Ip^S\left(1-\frac{\alpha \partial_{\alpha\alpha} g(\alpha,\theta)}{\partial_{\alpha} g(\alpha,\theta)}\right)-\pi p^S-p^S\theta\pi \frac{g(\alpha,\theta)}{\partial_{\alpha} g(\alpha,\theta)}\\
&\dot{p^I}=-\gamma p^I+rp^Ip^S\frac{\alpha \partial_{\alpha\alpha} g(\alpha,\theta)}{\partial_{\alpha} g(\alpha,\theta)}-rp^I(1-p^I) \\
&\dot{p^V}=rp^Ip^V+p^S\theta\pi \frac{g(\alpha,\theta)}{\partial_{\alpha} g(\alpha,\theta)}\\
\end{aligned}
\right.
\end{equation}
We do not write the dependence of the variables on time nor the degree of the node in order to avoid heavy notation.

We consider a cost function similar to the proposed in
 \cite{doncel2017mean,HETHCOTE1973365,LAGUZET2015180}, which contains a lineal term  in the vaccination rate $c_V\pi_t$
where $c_V$ may depend on the cost of the vaccine and its possible side effects; and a term
modeling the cost incurred by an infected patient per time unit, which could include the
possible loss generated by being unable to attend work, the costs of treatment and medical
consultation, and could consider the severity of an illness such as its sequels or even death.
So, the new individual wants to minimize their cost defined by:

\begin{equation}
\tilde{C_t}(\pi,\tilde{\pi})=\int_t^Tc_I \tilde{I}_s+c_V\tilde{\pi}_s\alpha_t^k\tilde{\theta}^{\xi(k)}ds.
\end{equation}

So, the new individual looks at the best response to strategy $\pi$, i.e., they want to play $BR(\pi)\in \arg \min_{\tilde{\pi}}\tilde{C_t}(\pi,\tilde{\pi})$. The minimum is
taken over
$$\Pi=\{\tilde{\pi}:[0,T]\to [0,\pi_{max}] \text{ bounded  and measurable}\},$$ which is a compact set for
the weak topology. This implies that $BR(\pi)$ is not empty, since any minimizing sequence
has a limit.

In Theorem 2 of \cite{doncel2016mean}, the authors show that if the dynamics and the costs
are continuous functions of the involved variables there always exists a mean-field
equilibrium in such games; thus the existence of a solution for our problem follows from their
result, since the cost is linear in $I$, the function $g$ is analytic, and the rates of transition
for the new individual depend linearly in $p^I$ or are constant for a fixed $\pi$.

Although this result guarantees the existence of a mean field equilibrium, we will compute the best response strategy for any $\pi$ played by the population
analyzing our problem as a continuous time Markov Decision Problem with finite horizon.

Denote $J_S(t)$, $J_I(t)$ the optimal cost starting at time $t$ in states susceptible and
infected, respectively. The optimal cost $J$ and the strategy $\tilde{\pi}^*$ that realize it,
satisfy the the following Hamilton-Jacobi-Bellman optimality equation
\cite{puterman2014markov}:
\begin{equation}\label{HJBeq}
\left\lbrace
\begin{array}{l}
J_S(T)=J_I(T)=0\\
-\dot{J_S(t)}=\inf_{\tilde{\pi}}\left[\tilde{\pi}\tilde{\theta}^{\xi(k)}(c_V-J_S(t))+ p_t^I r k (J_I(t)-J_S(t))\right]\\
-\dot{J_I(t)}=c_I-\gamma J_I(t)\\
\tilde{\pi}^*=\arg\min_{\tilde{\pi}\in \tilde{\Pi}}\left[\tilde{\pi}\tilde{\theta}^{\xi(k)}(c_V-J_S(t))+p_t^Irk(J_I(t)-J_S(t))\right].
\end{array}
\right.
\end{equation}
The equation explains how changes in the state, and the rates at which they occur, impact on the individual expected cost.
\begin{proposition}
Let $\tilde{\pi}^*$ be the strategy of a node with degree $k$ that realizes the optimal cost $J$. Then,
$\tilde{\pi}^*$ is threshold, that is, $\tilde{\pi}^*=\nu\mathds{1}_{[0,\tau]}(t)$ for some $\tau\in [0,T]$.
\end{proposition}

\begin{proof}

We will prove that the optimal strategy is constantly the maximum rate of vaccination until
some time $\tau$, and after that instant boils down to zero. Let us remark that the costs
associated to two different strategies that differs in a null measure set is the same.  So, we
have uniqueness up to a zero measure set.

We can see from the third equation in (\ref{HJBeq}) that
\[J_I(t)=\frac{C_I}{\gamma}(1-e^{\gamma(t-T)}).\] Hence $J_I$ decreases from
$J_I(0)=\frac{C_I}{\gamma}(1-e^{-\gamma T})$ to $J_I(T)=0$.

We also realize that if $J_S(t)>c_V$ then $\tilde{\pi}(t)=0$. Since $J_S(T)=0$ and the costs
are continuous, $\dot{J_S}(T)=0$ therefore, if we call $\tau$ the first instant at which $J_S$
is below $c_V$, we have $J_S(t)\leq J_I(t)$ for all $\tau\leq t\leq T$, such that the second
term in the second equation in (\ref{HJBeq}) is non-negative. If $J_S$ does not cross $c_V$
then we take $\tau=0$.

Moreover, if the cost of being susceptible is bigger than the vaccine cost, the derivative of
$J_S$ will be even smaller. Hence $J_S(t)\leq J_I(t)$ for all $0\leq t\leq T$ and therefore
$J_S$ is always decreasing before $\tau$. This concludes the proof.
\end{proof}

Let us remark that the optimal cost depends on the degree $k$ of the individual, thus $\tilde{\pi}_t^*=\tilde{\pi}_t^{*k}$. Moreover, if we suppose that $\xi(k)$ is increasing in $k$, meaning that nodes with more contacts have more incentive to vaccinate, the derivative of $J_S(t,k)$ is always bigger than the derivative of $J_S(t,j)$ when $k$ is bigger than $j$, and both $J_S$ and $J_I$ are equal at the beginning of the propagation of the disease if the initial proportion of infected is small enough. Thus, the threshold $\tau_k$ will be bigger than $\tau_j$. We have:

\begin{proposition}
Let $\pi_t^{*k}=\nu \mathds{1}_{[0,\tau_k]}(t)$ be the optimal strategy for an individual of degree $k$. Then $\tau_k>\tau_j$ if $k>j$.
\end{proposition}

\subsection{Social Optimum}

Now we consider a centralized planner trying to optimize from the point of view of the total population.
We first consider the control
system of the form $\dot{x}=\varphi(x,\pi)$ like (\ref{MyEq}) where the set $\Pi$ of
admissible controls is compact, and the family of admissible control functions $\pi$ is only
restricted by its measurability. Given the initial data $x(0)=x_0$ the Cauchy problem has a
unique solution, as we stated in Proposition \ref{flips}.

Given an initial data $(s,y)$ we consider the general optimization problem:
\begin{equation}\label{minimo}
minimize: \quad J(s,y,\pi)=\int_s^TL(x(t),\pi(t))dt+\Psi(x(T))
\end{equation}
where $L$ is the instant cost functional, $\Psi$ is the final cost, and the state variable $x$
depends not only on time but on the control and the initial data. The optimal $\pi$ is taken
over $\Pi$ the set of measurable functions $\pi:[0,T] \to[0,\nu]$.

As stated by the dynamic programming method, the optimal control can be characterized by
the value function $V(s,y):=\inf_{\pi \in \Pi} J(s,y,\pi)$, but the classical point of view does
not allow discontinuous control functions. Hence we start by verifying the hypothesis that
our setting must satisfy in order to ensure existence and uniqueness of the optimal control,
based on more general results on viscosity solutions theory \cite{bresan,trelat2008controle}.
In the case of measurable control, we can also apply the Pontryagin's Maximum Principle
with less restrictive assumptions.

According to Lemma 9.2 in \cite{bresan}, the functionals involved must satisfy
\begin{equation}\label{assumpPMP}
\begin{array}{ll}
\vert\varphi(x,\pi)\vert \leq C, &\quad \vert\varphi(x_1,\pi)-\varphi(x_2,\pi)\vert\leq C|x_1-x_2|, \\
|L(x,\pi)|\leq C, &\quad |\Psi(x)|\leq C, \\
\vert L(x_1,\pi)-L(x_2,\pi)\vert\leq C|x_1-x_2| &\quad \vert \Psi(x_1)-\Psi(x_2)\vert\leq C|x_1-x_2|,
\end{array}
\end{equation}
for all $x_1,x_2 \in \R^7$, and $\pi \in \Pi$, for some constant $C$. Under these
assumptions, the value $V$ is a bounded, Lipschitz continuous function, and it can be
characterized as the unique viscosity solution to a Hamilton-Jacobi equation.

Since the epidemics dynamics satisfy these assumptions on $\varphi$, to the best of our
knowledge, the most general condition on the cost functions to admit a solution is to be
Lipschitz in the state variable and bounded, which allows us to model a wide range of
real situations.

As a particular case inspired by the individual optimization problem exposed above and in
agreement with \cite{galvani2007long,HETHCOTE1973365,doncel2017mean,LAGUZET2015180}, we define the cost:
\begin{equation}
L_1(x,\pi)=c_II_t+c_V\pi_tg(\alpha_t,\theta_t).
\end{equation}

We can easily check,
by basic computations and
bounding the second term using the regularity of $g$ and the Mean Value Theorem, that
our setting satisfies  hypotheses  (\ref{assumpPMP}).

Further, given the data $x(0)=x_0$, let $t \mapsto x^*(t)=x(t,\pi^*)$  be an optimal trajectory
corresponding to the optimal control $\pi^*$. Following Theorems 7.18 and 11.27 in
\cite{trelat2008controle}, there exists an absolutely continuous application $t \mapsto p(t)\in
\R^7$ called the adjoint vector, and a real number $p_0\geq 0$, such that $(p,p_0)$ is non
trivial, and such that for almost every $t\in[0,T]$
\begin{equation}\label{pmp}
\begin{aligned}
&\dot{x^*}=f(x^*,\pi^*),\\
&x(0)=x_0,\\
&\dot{p^*}=-\frac{\partial H}{\partial x}(x^*,p^*,\pi^*),\\
&p^*(T)=0,\\
&\pi^*=argmin_\pi H(x^*,p^*,\pi),
\end{aligned}
\end{equation}
where the Hamiltonian of the system is $H=p_0L_1+pf$.

Summarizing, we have the following result.

\begin{proposition}
Let $\pi^*$ the strategy that minimizes \eqref{minimo} for the cost functions defined above. Then
$\pi^*$ is threshold.
\end{proposition}
\begin{proof}
Writing the equation for $\pi^*$ we get
\begin{equation}
\begin{aligned}
\pi^*= &argmin\Big\{\Big(p_0c_Vg(\alpha^*,\theta^*)-p_2^*\theta^*+
\qquad
\\ &\qquad \qquad +p_4^*\theta^*\partial_{\theta}g(\alpha^*,\theta^*) -p_5^*{p^S}^*+({p_7}^*-{p_5}^*){p^S}^*\theta^* {p^S}^* \frac{\partial_{\alpha \theta}(\alpha^*,\theta^*)}{\partial_\alpha g(\alpha^*,\theta^*)} \Big)\pi\Big\}
\end{aligned}
\end{equation}
which is a linear function of $\pi$ with principal coefficient $\rho^*$. Since we are minimizing over $\pi \in [0,\nu]$ we can
conclude that
 \begin{equation}\label{piSocialthreshold} \pi^*_t=\left\lbrace \begin{aligned}\nu
\hspace{20pt} \text{ if } \rho^*(t) < 0 \\ 0 \hspace{20pt} \text{ if } \rho^*(t)>0,
\end{aligned}\right.
\end{equation}
and the proof is finished.
\end{proof}

Since it is impossible to solve analytically the system (\ref{pmp}), the method of
Forward-Backward Sweep presented in \cite{RKBK} could be useful to understand the behavior of
$\rho^*$. From (\ref{assumpPMP}) we can deduce that $\rho^*(T)>0$ indicating that the vaccination rate must be zero after some moment.
As in the preceding section, the optimal vaccination strategy is of threshold type, in the sense
that vaccination must be intended with the maximum effort (maximum rate) and otherwise discouraged. This result may be explained by the intuition that the individuals would vaccinate if the vaccination cost is lower than the potential cost associated to the infection, and the probability of contracting the disease decrease after the peak of the infected population. However, remark that a quantitative characterization of the change of regime depends on the complete behavior of the epidemics. 
This is solved by calculating Bellman equations and the theory of Dynamic Programming.

\section{Phenomenological Conclusions and Epidemic Analysis}

Given a cost and a maximum vaccination rate, our results can be exploited to design
vaccination policies which combined effectiveness to get immunization of the population and
mild economical cost. Note that it is natural to suppose that there exists an upper bound on
the vaccination rates which takes into account both economical and organizational
considerations.

Knowing the effective contact rate $r$ of the disease in consideration and its recovering rate
$\gamma$, the vaccination budget represented in $\nu$, and taking in account the
connectivity of the population, we aim at finding this optimal vaccination effort.

In usual SIR mean field models,
there exist two regimes: above a threshold in the ratio $\frac{r}{\gamma}$ the epidemic will
propagate and the final number of infected reach a fraction of the population even when the
number of initially infected $\varepsilon$ is arbitrary small. Below this threshold the final
size of epidemic will be proportional to $\varepsilon$.

Here, this threshold can be described in terms of the connectivity of the graph, regarding both the
average number of contacts and the size biased distribution. In our model, the number of
newly infected in a small time interval is proportional to the quantity $p^I$. The threshold is
then determined by the following equation:
\[0=\dot{\p{I}{t}}|_{t=0}=(-\gamma\p{I}{t}+\p{I}{t}r\p{S}{t}\frac{\alpha_t \gxx{t}}{\gx{t}}-r\p{I}{t}(1-\p{I}{t}))|_{t=0},\]
If we take an $\varepsilon$-proportion of initially infected individuals, then we have the initial
conditions:
\begin{equation}\label{initialConditions} I_0=\varepsilon, \quad S_0=1-\varepsilon,  \quad
p_0^I=\frac{\varepsilon}{1-\varepsilon}, \quad  \mbox{ and }  \quad p^S_0=\frac{1-2\varepsilon}{1-\varepsilon}.
\end{equation}
Assuming $\varepsilon \ll 1$, after a simple
computation we get the inequality:
\begin{equation} \label{thresh}
r\left( \frac{\partial_{\alpha\alpha}g(1,1)}{\partial_\alpha g(1,1)}-1 \right) >\gamma
\end{equation}

As we mentioned before, $\frac{\partial_{\alpha\alpha}g(1,1)}{\partial_\alpha g(1,1)}$ is the expectation of the size biased
distribution, therefore it represents the mean degree of a random neighbor (a new infected
agent). By subtracting one  (the spread will not return to the infecting agent), we get the
expected out degree, which represent the possible number of contacts that the newly
infected individual has. Therefore, equation (\ref{thresh}) states that the epidemic will occur if the total
rate of infection is bigger than the rate of recovering. We can rewrite it and we find back the
already known critical threshold stated in Prop 6.1 of \cite{RevModPhys}:
\[R_0=\frac{r}{r+\gamma}\frac{\partial_{\alpha\alpha}g(1,1)}{\partial_\alpha g(1,1)}>1 \]
Written in terms of the transmissibility $T=\frac{r}{r+\gamma}$ (i.e., probability of
infection given contact between a susceptible and infected individual) we recover the
formula stated in \cite{PhysRevE.66.016128} for the epidemic threshold.

Finally, for the optimal strategies that we found in the previous section, we compute the
impact of the vaccination process in terms of epidemic sizes, showing that the decrease of
the susceptible population can be described through the generating function of the network,
the proportion of IS edges and the vaccination rate, through a simple formula.

Following Miller \cite{miller2011note}, we can compute the final epidemic size in our model
which allows to measure the  impact  of the vaccination process. To that end, we can
compute $\alpha_\infty^k$ the probability that a randomly chosen susceptible node $u$ with degree
$k$ is never infected in terms of the transmissibility $T=\frac{r}{r+\gamma}$, interpreted here as the probability that one of the edges of $u$ with an infected neighbor transmit the disease to $u$.
The probability that this neighbor is never infected, given that $u$ does not transmitted the disease, is $T\left(1-\frac{\gx{\infty}}{\partial_\alpha g(1,1)}\right)$ because the neighbor degrees follow the size biased distribution. Thus, the probability that the edge is not an infected contact for $u$ solves the fixed point equation:
\begin{equation}
\alpha_\infty=1-T+T\frac{\gx{\infty}}{\partial_\alpha g(1,1)}
\end{equation}

\begin{proposition}
If $\pi_t=\nu\mathds{1}_{[0,\tau]}(t)$ and $\xi(k)=k$, taking $$\lambda(\tau)=\frac{\int_0^\tau\pi_t dt}{\int_0^\tau\pi_t + rp^I_t dt},$$ we have: \[R_\infty=S_0-\psi(\alpha_\infty e^{-\tau\nu})+\lambda(\tau)(g(\theta_\tau \alpha_\tau)-1).\]
\end{proposition}

\begin{proof}
After some
simple computations we get
\[\theta_\infty=e^{-\tau\nu}.\]
Now, we can also compute the probability that an initially susceptible node remains
susceptible at the end of the epidemics, corresponding to never been infected nor vaccinated:
\[S_\infty(k)= \theta_\infty^{\xi(k)}\alpha_\infty^k=\alpha_\infty^k e^{-\tau\nu}. \]
With this, \[S_\infty=\sum_k \mu_0^S(k)\theta_\infty^{\xi(k)}\alpha_\infty^k=g(\alpha_\infty,e^{-\tau\nu}).\]
In order to compute the size of the epidemic (total number of recovered agents), first we compute the final number of vaccinated agents. This can be thought of as an exponential race with non-homogeneous rate. Thus, the probability that a node of degree $k$ to be vaccinated will be:
\[V_\infty(k) = \frac{\lambda^\pi_k(\tau)}{\lambda^\pi_k(\tau)+\lambda^I_k(\tau)}(1-e^{-\tau\nu\xi(k)}\alpha_\tau^k),\]
where $\lambda^\pi_k(t)=\int_0^t \xi(k)\pi_s ds$ and $\lambda^I_k(t)=\int_0^t rkp^I_s ds$.
With this, $V_\infty=\sum_k \mu_0^S(k) V_\infty(k)$,
and therefore the final epidemic size is
$R_\infty=1-S_\infty-V_\infty$.
The case of the degree-proportional vaccination, $\xi(k)=k$, can be described by the generating function of the initial degree of the susceptible. By taking $$\lambda(\tau)=\frac{\lambda^\pi_k(\tau)}{\lambda^\pi_k(\tau)+\lambda^I_k(\tau)}=\frac{\int_0^\tau\pi_t dt}{\int_0^\tau\pi_t + rp^I_t dt}$$ one can write:
\[V_\infty=\sum_k \lambda(\tau)[1-(\theta_\tau\alpha_\tau)^k]\mu_0^S(k)=\lambda(\tau)S_0-\lambda(\tau)\psi(\theta_\tau\alpha_\tau).\]
On the other hand, $\alpha_\infty$ is the fixed point of $\alpha_\infty=1-T+T\psi'(\alpha_\infty e^{-\tau\nu})/\psi'(1)$ and $S_\infty=\psi(\alpha_\infty e^{-\tau\nu})$.
This finishes the proof.
\end{proof}

Here, we can see the strong dependence on the connectivity model of the network through
the generating function $g$ (or $\psi$) and on the maximum
rate of vaccination in order to reduce, (exponentially in $\tau \nu$ and through the functions $g$ and $\psi$ respectively), the propagation of the epidemic.
Additionally, the maximum rate of vaccination can be translated in the budget of the decision
maker, because it may indicate how effective the decision to vaccinate may be.

\subsection{Simulations}\label{lasecdesimu}

Currently, the development of vaccines is usually subsequent to the appearance of a possible
epidemic. In many of them, it is through the value $R_0$ that the epidemiological parameter
$r$ (or $\beta$ for homogeneous compartmental models) can be estimated. It contains the
probability that in an encounter between a susceptible and an infected person, the former
becomes ill, and the number of encounters in a unit of time; and $\gamma$ which is related
to the average time that a person is infectious. Knowing these parameters, the quotient $\frac{\partial_{\alpha\alpha}g(1,1)}{\partial_\alpha g(1,1)}$ can be inferred from the formula (\ref{R0}).

In this section we fix the parameters of the epidemics and simulate the propagation of the
disease and the vaccination process in the cases $\xi(k)=ak+b$ solving numerically the system of equations (\ref{eqConst}).

We suppose a threshold vaccination strategy of the form $\pi_t=\nu\;\mathds{1}_{[0,\theta]}$
in the period of time $[0,T]$ thus the global optimization cost is $$C(\pi)=C(\tau)=\int_0^T
c_II_t+c_V\nu\mathds{1}_{[0,\tau]} \phi \psi(\beta_t) dt.$$ We do not write explicitly the
dependence of the variables on $\tau$ to reduce notation. Let us remember that an agent of
degree $k$ will vaccinate according rate $ak+b\nu$ during the vaccination period $[0,\tau]$,
which is hidden in the probability generating function $g$.

Hence, the optimization problem reduces to find $\tau^*=\arg\min_\tau C(\tau)$. Thus, for each $\tau$ we run the numerical integration and select the optimal threshold that minimize this cost. 

We do this for two vaccination policies, constant and degree-proportional. In order to compare both strategies, we set $a=1$ and $b=\psi'(1)$, the mean number of neighbors, obtaining the same vaccination rate.

As we stated in Section \ref{meanfield}, the mean field model corresponds to take a
Poisson degree distribution. In \cite{brauer2012mathematical}, the authors establish that
Poisson distribution is not realistic for the modeling of contacts, and propose the use of a
bimodal distribution in which a proportion $p$ of the population follows a Poisson law of
mean $C$ and the rest a delta-$L$ distribution. If the whole population follows a delta
distribution the associated graph is called regular. Other distribution were also proposed in
\cite{Herrmann2020.04.02.20050468,MILLER2018192,PhysRevE.66.016128}, in particular the
power law, found in several social networks
\cite{smallWorldNet,liljeros2001web,lloyd2001viruses,pastor2001epidemic}.

Thus, we consider four different networks of size $N=10000$, associated to a correspondent
generating function of the degree distributions, the four with mean degree $\psi'(1)=5$:
\begin{enumerate}
    \item[(a)] Poisson: Degrees are distributed according a random variable Poisson $\mathcal{P}(\lambda)$ with parameter $\lambda=5$, which gives
        $\psi(z)=e^{5(z-1)}$.
    \item[(b)] Bimodal: a proportion $p=4/5$ has degree
    $\mathcal{P}(3)$ and the rest follows a
        delta distribution with parameter $13$, thus $\psi(z)=pe^{3(z-1)}+(1-p)z^{13}$.
    \item[(c)] Regular: all nodes have the same degree $5$, $\psi(z)=z^5$.
    \item[(d)] Power Law: the probability that a node has degree $k$ is given by
        $p_k=\frac{k^{-\alpha}e^{-k/\kappa}}{Li_\alpha(e^{-1/\kappa})}$ for $k$ in $\N$,
        $\alpha=1.474$ and $\kappa=100$, resulting
        $\psi(z)=\frac{Li_\alpha(z)(ze^{-1/\kappa)}}{Li_\alpha(e^{-1/\kappa})}$ where $Li_s(z)$ is
        the $s-$polylogarithm of $z$. We consider an exponential cut-off around $\kappa=100$ in order to have finite 	moments.	
\end{enumerate}

\begin{figure}[htb]
  \begin{subfigure}[b]{0.5\linewidth}
    \centering
    \includegraphics[width=0.75\linewidth]{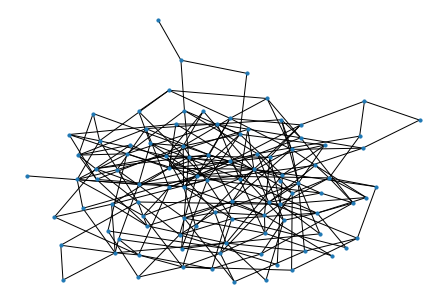}
    \caption{Poisson}
    \label{fig7:a}
    \vspace{4ex}
  \end{subfigure}
  \begin{subfigure}[b]{0.5\linewidth}
    \centering
    \includegraphics[width=0.75\linewidth]{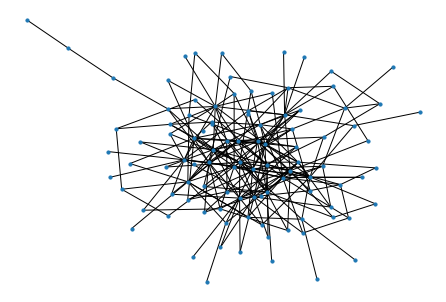}
    \caption{Bimodal}
    \label{fig7:b}
    \vspace{4ex}
  \end{subfigure}
  \begin{subfigure}[b]{0.5\linewidth}
    \centering
    \includegraphics[width=0.75\linewidth]{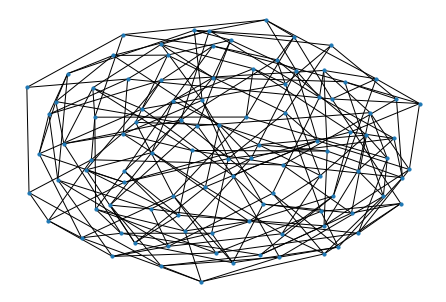}
    \caption{Regular}
    \label{fig7:c}
  \end{subfigure}
  \begin{subfigure}[b]{0.5\linewidth}
    \centering
    \includegraphics[width=0.75\linewidth]{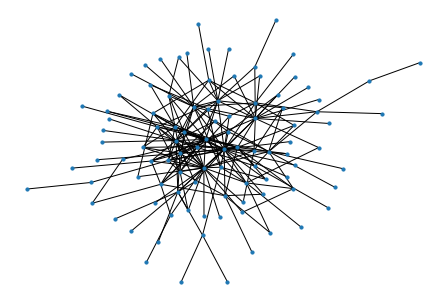}
    \caption{Power Law}
    \label{fig7:d}
  \end{subfigure}
  \caption{Instances of graphs when N=100 for the same mean degree $\psi'(1)=5$.}
  \label{grafos}
\end{figure}

Here we can see that the graph structure makes a difference
for the propagation or eradication of the
disease through the existence of hubs that could play the role of
super-spreaders, or almost disconnect the networks due to vaccination at higher rate. We summarize some network measures
of  the
generated graphs, and show the numerical solutions of system \eqref{MyEqZ} in each of them in order to compare
the spread of the epidemic. We use the \texttt{python} package
\texttt{networkx} for the graph creation and the calculus of the coefficients
\cite{hagberg2008exploring}, and we solve the system  \eqref{MyEqZ} using a Runge-Kutta 4.

We can see in the draw of the graph exposed in Figure \ref{grafos}, that the Bimodal network has more nodes of high degree than
a Poisson, and less than a Power Law network, according the results in
\cite{brauer2012mathematical}.

 For
the epidemic evolution we take $r=3$ and $\gamma=1$, and in  Table \ref{tab:R0} we can see that networks
(b) and (d) have a bigger basic reproduction number due to the size biased distribution and
the expected excess degree.

\begin{table}[htb]
\begin{center}
 \begin{tabular}{|l| c c|}
 \hline
 Graph & $\psi''(1)$ & $R_0 $\\ [0.5ex]
 \hline\hline
 (a) Poisson & 25 & 3.75  \\
 \hline
 (b) Bimodal & 61.5 & 5.76  \\
 \hline
 (c) Regular & 20 & 3.0  \\
 \hline
 (d) Power-Law & 80.31 & 12.03  \\
 \hline
\end{tabular}
\caption{Network quantities}
\label{tab:R0}
\end{center}
\end{table}

Typical parameters of the nodes of social networks are the clustering coefficient, which indicate the tendency of a node to form a cluster,
and the betweenness
centrality, measuring the presence of the node in shortest paths between different nodes, see  \cite{transivity}.
Both play an important role in the
propagation of a disease, the former due to potential contagion to larger groups of nodes, and the later because imply a quick connection
between different parts of the network. Also, the closeness coefficient measures how far apart are the nodes, see \cite{borgatti2005centrality}.
Proper definitions of these coefficients could be:
\begin{itemize}
\item Betweenness centrality of a node $v$ is the sum of the fraction of all-pairs shortest paths that pass through $v$.
\item The density for undirected graphs is $\frac{2m}{n(n-1)}$,
where $n$ is the number of nodes and $m$ is the number of edges.
\item Closeness centrality of a node $v$ is the reciprocal of the average shortest path distance to $v$ over all $n-1$ reachable nodes.
\item Degree centrality is defined as the number of links incident upon a node.
\end{itemize}
See \cite{newman2003structure} for more details.

In Table \ref{tab:cent} we show the average number for these measures after  several experiments with networks of $10^4$ nodes. We
 observe that density, related to the mean degree of the four networks, is similar. Nevertheless there are important differences
 between the averages of the clustering coefficients.

\begin{table}[htb]
\begin{center}
 \begin{tabular}{|l| c c c c|}
 \hline
 Graph & Betweenness  & Density & Clustering & Closeness\\
  & \footnotesize{$\times 10^4$ }& \footnotesize{$\times 10^4$ } & \footnotesize{$\times 10^4$ } &   \\
 \hline\hline
 (a) Poisson & 4.83 & 5.00 & 4.80 & 0.168 \\
 \hline
 (b) Bimodal & 3.77 & 5.06 & 11.46 & 0.18 \\
 \hline
 (c) Regular & 5.36 & 5.00 & 3.54 & 0.157 \\
 \hline
 (d) Power Law & 3.34 & 5.03 & 28.69 & 0.20 \\
 \hline
\end{tabular}
\caption{Average graph parameters.}
\label{tab:cent}
\end{center}
\end{table}

In Figure \ref{CostVariation}, we present three simulations for a Bimodal graph for different
optimization costs and vaccination rates, for the constant (abbreviated Const) and degree proportional (abbreviated Deg) vaccination strategies.
We plot the total cost, the final number of vaccinated and the epidemic size as a function of
the vaccination threshold $\tau$, obtaining the 
same behavior and shape in all the
studied cases, namely, the final vaccinated population is increasing in $\tau$, and the costs decrease up to the optimum.

\begin{figure}[htb]
    \centering 
\begin{subfigure}{0.3\textwidth}
  \includegraphics[width=\linewidth]{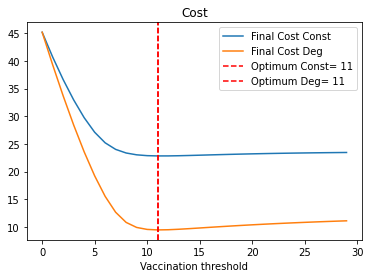}
\end{subfigure}\hfil 
\begin{subfigure}{0.3\textwidth}
  \includegraphics[width=\linewidth]{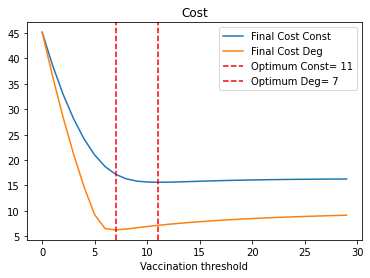}
\end{subfigure}\hfil 
\begin{subfigure}{0.3\textwidth}
  \includegraphics[width=\linewidth]{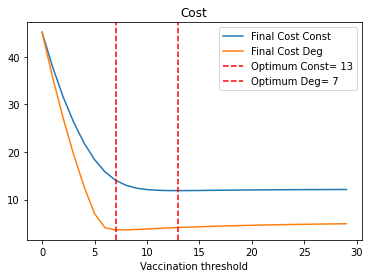}
\end{subfigure}
\medskip
\begin{subfigure}{0.3\textwidth}
  \includegraphics[width=\linewidth]{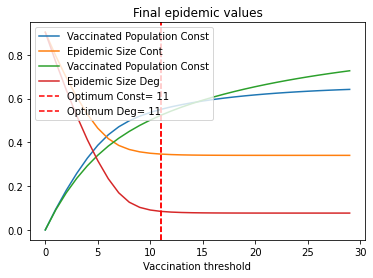}
  \caption{$c_V$=10, $c_I$=50, $\nu$=0.2}

  \label{fig:4}
\end{subfigure}\hfil 
\begin{subfigure}{0.3\textwidth}
  \includegraphics[width=\linewidth]{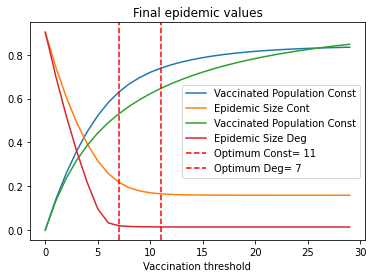}
  \caption{$c_V$=10, $c_I$=50, $\nu$=0.3}
  \label{fig:5}
\end{subfigure}\hfil 
\begin{subfigure}{0.3\textwidth}
  \includegraphics[width=\linewidth]{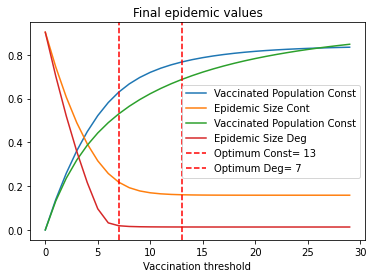}
  \caption{$c_V$=5, $c_I$=50, $\nu$=0.3}
  \label{fig:6}
\end{subfigure}
\caption{Threshold optimization for Bimodal Network. Const is for constant vaccination and Deg for degree-proportional. $c_V$ and $c_I$ are the vaccination and infection costs, and $\nu$ is the vaccination rate.} \label{CostVariation}
\end{figure}

We can see from the simulation of the optimization in Figure \ref{CostVariation} that the
epidemics size decreases notably with respect to the case without vaccination which
corresponds to the vaccination threshold $\tau=0$. On the other hand, we observe that the
epidemics size curve flattens around the value $\tau^*$ that minimizes the cost. Therefore, it
makes no sense (neither financially nor epidemiologically) to continue a vaccination program
after that time, because the final cost will be bigger and the population infected will not
decrease further with more vaccinated individuals.

In Figure \ref{evolution}, we present the result of the simulation for the variables $S$,
$I$, $R$ and $V$ when $\nu=0.2$, $c_V=10$ and $c_I=50$, and the vaccination threshold is obtained by solving the optimization problem.

\begin{figure}[htb]
  \begin{subfigure}[b]{0.5\linewidth}
    \centering
    \includegraphics[width=0.75\linewidth]{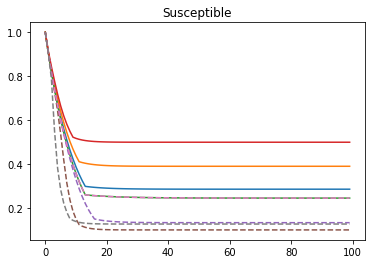}
    \caption{Susceptible over time}
  \end{subfigure}
  \begin{subfigure}[b]{0.5\linewidth}
    \centering
    \includegraphics[width=0.75\linewidth]{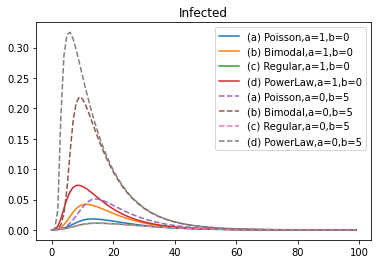}
    \caption{Infected over time}
  \end{subfigure}

  \begin{subfigure}[b]{0.5\linewidth}
    \centering
    \includegraphics[width=0.75\linewidth]{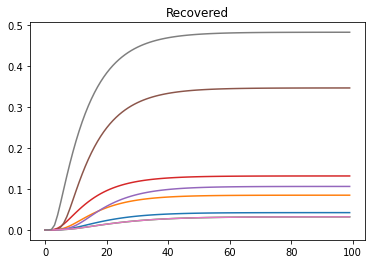}
    \caption{Recovered over time}
  \end{subfigure}
  \begin{subfigure}[b]{0.5\linewidth}
    \centering
    \includegraphics[width=0.75\linewidth]{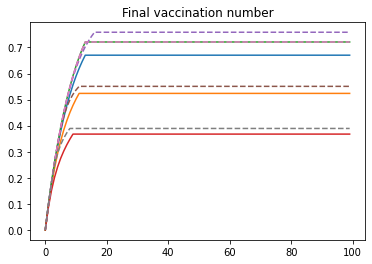}
    \caption{Vaccinated over time}
  \end{subfigure}

  \caption{Evolution of the epidemics indicators. Numerical integration of system \ref{eqConst} with $\xi(k)=ak+b$, solid lines are for $\xi(k)=k$ (degree dependent vaccination) and dashed for $\xi(k)=5$ (constant vaccination) in networks with the same mean degree equal to 5.}
   \label{evolution}
\end{figure}

We expect a faster
spread of the disease in the networks with a bigger clustering and closeness coefficients and
lower betwenness, from the ideas exposed in the last paragraph.

Comparing Figure \ref{evolution} with Tables \ref{tab:R0} and \ref{tab:cent} we can check
that the velocity of propagation of the disease and the maximum number of infected agents
is bigger in networks (b) and (d), where the basic reproduction number and the clustering and
closeness coefficients are bigger and the betwenness is smaller. Moreover, the size of the
epidemic in (d) is greater than the size for networks (a) and (c).

On the other hand, the final number of vaccinated is bigger in networks (a) and (c), almost
doubling  the number of vaccinations in network (d). The former are the most homogeneous in
terms of degree and mixing modeling, meaning that the vaccination rate of different
individuals are similar, while (d) have a big number of degree 1 individuals and a
considerable proportion of high-degree nodes, whose vaccination rates have great variability.

The bimodal and power law networks seem to best reflect the interactions between people
\cite{smallWorldNet,liljeros2001web,lloyd2001viruses,pastor2001epidemic}. In those cases,
we observe that the spreading of the disease is faster and infect a bigger proportion of the
network, due the presence of hubs, and the degree-dependent vaccination effect is not
enough to contain the epidemic

Let us note that for networks (b) and (d) we have $R_0\sim 10$, and the herd immunity is
reached when about $50\%$-$60\%$ of the population is infected or vaccinated. In a
complete network, we need about $90\%$ of the population recovered or vaccinated. Also,
for $R_0\sim 3-4$, we need about of $70\%$-$75\%$ in the complete network, while in
networks (a) and (c) the vaccinated and recovered agents  surpass the $90\%$ of the
population.

Finally, we present the results of the numerical integration for four networks with the same $R_0=3.75$. The degree distribution is the same as the used in the previous graphics, but with a proper selection of the parameters: $\lambda=5$ for the Poisson, $\lambda=3,L=8,p=0.73$ for the Bimodal, the degree of the Regular network is $6$, and for the Power Law we take $\alpha=2$ with exponential cutoff $\kappa=20$.

\begin{figure}[htb]
  \begin{subfigure}[b]{0.5\linewidth}
    \centering
    \includegraphics[width=0.75\linewidth]{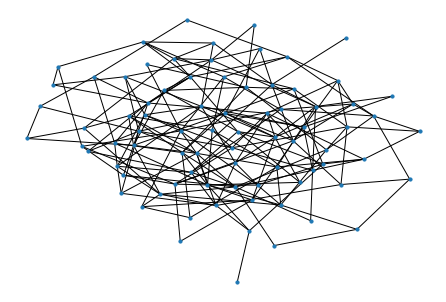}
    \caption{Poisson}
    \vspace{4ex}
  \end{subfigure}
  \begin{subfigure}[b]{0.5\linewidth}
    \centering
    \includegraphics[width=0.75\linewidth]{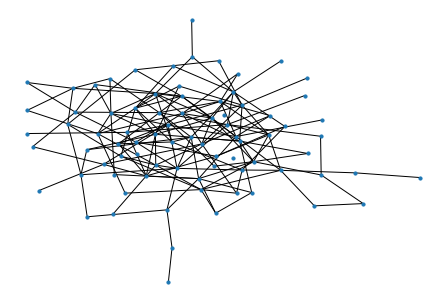}
    \caption{Bimodal}
    \vspace{4ex}
  \end{subfigure}
  \begin{subfigure}[b]{0.5\linewidth}
    \centering
    \includegraphics[width=0.75\linewidth]{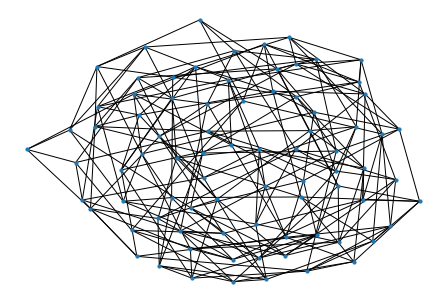}
    \caption{Regular}
  \end{subfigure}
  \begin{subfigure}[b]{0.5\linewidth}
    \centering
    \includegraphics[width=0.75\linewidth]{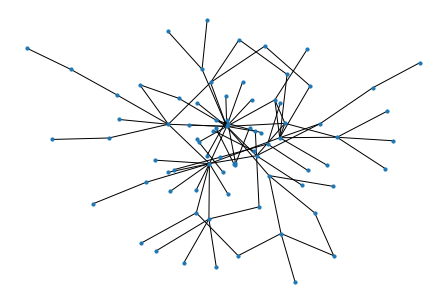}
    \caption{Power Law}
  \end{subfigure}
  \caption{Instances of graphs when N=80 for the same $R_0$.}
  \label{grafosmismoR0}
\end{figure}

In this case, the Power Law (d) network is very different to the other three. Its coefficients of centrality indicate that a more rapid spread of the disease is expected. Here, the role of the hubs is more relevant because the quotient $\psi''(1)/\psi'(1)$ is the same. This  is shown in the Table \ref{tab:centMismoR0}. We can also expect from Figure \ref{grafosmismoR0} that vaccinating the hubs, which will occur at high rate, will somehow cleave the network, stopping the propagation of the disease.

Effectively, in Figure \ref{evolutionMismoR0} we observe that the final size of the epidemic is substantially lower than in the other networks, and the final number of vaccinated is also much lower, indicating that in this case the network used to model the reality plays a central role.

From the presented instances, we can also observe that acquaintance vaccination is more effective in all the cases, arriving to lower numbers of infected and with a less number of vaccinated, with the exception of the case (c) of the regular network, where both vaccination strategies perform the same because the degree is the equal for all the nodes.

\begin{figure}[htb]
  \begin{subfigure}[b]{0.5\linewidth}
    \centering
    \includegraphics[width=0.75\linewidth]{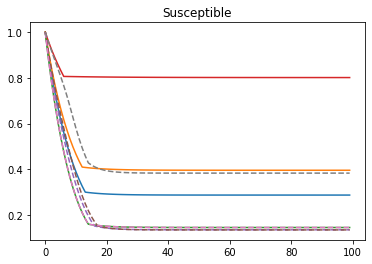}
    \caption{Susceptible over time}
  \end{subfigure}
  \begin{subfigure}[b]{0.5\linewidth}
    \centering
    \includegraphics[width=0.75\linewidth]{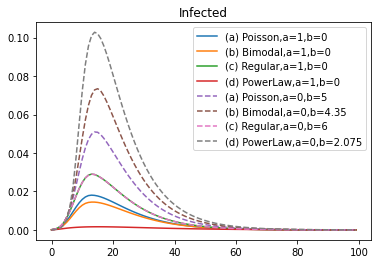}
    \caption{Infected over time}
  \end{subfigure}

  \begin{subfigure}[b]{0.5\linewidth}
    \centering
    \includegraphics[width=0.75\linewidth]{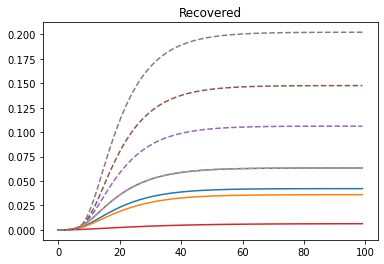}
    \caption{Recovered over time}
  \end{subfigure}
  \begin{subfigure}[b]{0.5\linewidth}
    \centering
    \includegraphics[width=0.75\linewidth]{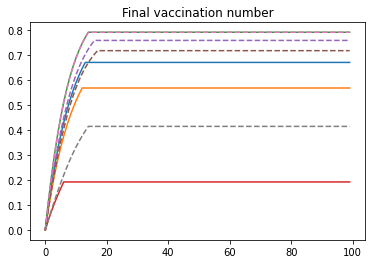}
    \caption{Vaccinated over time}
  \end{subfigure}
  \caption{Evolution of the epidemics indicators for the same $R_0$. Numerical integration of system \ref{eqConst} with $\xi(k)=ak+b$, solid lines is for $\xi(k)=k$ (degree proportional rate of vaccination) and dashed for $\xi(k)=1$ (constant rate of vaccination).}
   \label{evolutionMismoR0}
\end{figure}

\begin{table}[htb]
\begin{center}
 \begin{tabular}{|l| c c c c|}
 \hline
 Graph & Betweenness  & Density & Clustering & Closeness\\
  & \footnotesize{$\times 10^4$ }& \footnotesize{$\times 10^4$ } & \footnotesize{$\times 10^4$ } &   \\
 \hline\hline
 (a) Poisson & 3.47 & 4.88 & 4.92 & 0.219 \\
 \hline
 (b) Bimodal & 3.29 & 4.38 & 4.45 & 0.199 \\
 \hline
 (c) Regular & 3.19 & 5.99 & 4.22 & 0.239 \\
 \hline
 (d) Power Law & 2.24 & 2.18 & 5.68 & 0.109 \\
 \hline
\end{tabular}
\caption{Average of centrality coefficient for the same quotient $\psi''(1)/\psi'(1)$.}
\label{tab:centMismoR0}
\end{center}
\end{table}

\section{Conclusion}

We considered a generalization of the SIR model on a large configuration model adding a
vaccination strategy. The variables describing the dynamics were determined by a Markovian
contact process on a diverse population where individuals have different degrees in a
network of random connectivity, the vaccination mechanism depending on the number of
possible contacts of an individual.

We derived large graph limits for the evolution of the epidemics in this context and using
a function $g$ which is a variation of the initial susceptible population we derived a
system of seven equation that describe the SIRV dynamics. We analyze the particular case where the rate is linear in degree involving the generating probability function of the initial degree distribution $\psi$.

We then solved the associated optimal control problems for a general vaccination rate, proving existence and uniqueness
of the solutions under mild assumptions on cost functions. We also characterized the optimal
solutions as threshold type.

For this type of strategies, we computed the impact of the vaccination process in terms of
epidemic sizes, finding a fixed point equation that describes the impact of the vaccination intensity and the time on which it is developed through the function $g$ (or $\psi$) and depending on the proportion of $IS$ edges and the
vaccination rate.

Finally we studied four particular networks and their centrality coefficients, relating it with
the epidemic indicator $R_0$, also depending on $g$. Given a maximum vaccination rate and
for fixed disease parameters, we solved numerically the system of equations for the cases of degree-proportional and constant vaccination, on networks with the same mean degree first and then with the same mean excess degree. We
observed that in the networks that describe better the interaction of the individuals, the
epidemics level of infection can be significantly different than in homogeneous models.


\section{Proofs}\label{pruebas}

\subsection{Poisson Point Measures and Stochastic Differential Equations}
Inspired by \cite{decreusefond2012large} and \cite{fournierMeleard04} we will represent the
behavior of our dynamic as a process which is solution of a system of a stochastic
differential equations derived from Poisson Point Measures (PPM).

We will use three different PPM for each event which modifies the quantities we are
interested in: an infection, a recovery, or a vaccination. We need to identify the rates of this
events and how to update the measures on the graph.

Suppose an event occur at time $T$, and let us analyze the first case, an infection. For that, it
is convenient first to consider
\begin{equation}
\lambda_{T-}(k)=rk\frac{N^{IS}_{T-}}{N^S_{T-}},
\end{equation}
the rate of infection of a given $k$-degree individual at time $T$. They will have their
half-edges connected according to the quantities $\mu_T$ and distributed following a
multivariate hypergeometric distribution. We denote
\begin{equation}
p_{T-}(j,l,m\mid k-1)=\frac{ {N^{IS}_{T-}-1 \choose j-1}{N^{RS}_{T-} \choose l}{N^{VS}_{T-} \choose m}{N^{S}_{T-}-N^{RS}_{T-}-N^{IS}_{T-}-N^{VS}_{T-} \choose k-1-j-l-m}}{{N^{S}_{T-}-1 \choose k-1}}.
\end{equation}
Finally, given $k$, $j$, $l$ and $m$, we have to update the measures $\muis{T}$, $\murs{T} $
and $\muvs{T}$ choosing the infected, recovered and vaccinated individuals who will be
connected to the newly infected. In order to do that, we draw three vectors
$u=(u_1,...,u_{I_{T-}})$, $v=(v_1,...,v_{R_{T-}})$, and $w=(w_1,...,w_{V_{T-}})$ indicating how
many links each $I$, $R$ or $V$ node has with the newly infected. We consider
$\mathcal{U}=\bigcup_{n\in \mathbb{N}}(\mathbb{N}_0)^n$ and for each $\mu\in\M$ and
$n\in \N$ we define $$\U\supseteq \U(\mu,n):=\Big\{u=(u_1,...,u_{\pint{\mu}{\mathds{1}}})\ :
\sum_{i=1}^{\pint{\mu}{\mathds{1}}}u_i=n \text{ and } u_i \leq \zeta_i(\mu) \Big\},$$ where
$\zeta_i(\mu):=F^{-1}_\mu(i)$ is the degree according $\mu$ of the $i$-th node. In a similar way we define $v$, $w \in \mathcal{U}$ depending on the measures $\murs{T}$ and $\muvs{T}$.

Thus the
number of edges of type $IS$, $RS$ or $VS$ will be given respectively by
\begin{equation}
\begin{aligned}
\rho(u\mid j+1,\muis{T-})=&\frac{\prod_{i=1}^{I_{T-}}{\zeta_i(\muis{T-}) \choose u_i}}{{N^{IS}_{T-} \choose j+1}}\mathbf{1}_{u\in \U(\muis{T-},j+1)}, \\
\rho(v\mid l,\murs{T-})=&\frac{\prod_{i=1}^{R_{T-}}{\zeta_i(\murs{T-}) \choose v_i}}{{N^{RS}_{T-} \choose l}}\mathbf{1}_{u\in \U(\murs{T-},l)}, \\
\rho(w\mid m,\muvs{T-})=&\frac{\prod_{i=1}^{V_{T-}}{\zeta_i(\muvs{T-}) \choose w_i}}{{N^{VS}_{T-} \choose m}}\mathbf{1}_{w\in \U(\muvs{T-},m)}.
\end{aligned}
\end{equation}

We define
\[D(t,u,\mu)=\sum_{i=1}^{\pint{\mu_{t}}{\mathds{1}}}\delta_{\zeta_i(\mu_{t})-u_i}-\delta_{\zeta_i(\mu_{t})}\]
and
\[D_f(t,u,\mu)=\sum_{i=1}^{\pint{\mu_{t}}{\mathds{1}}}f(\zeta_i(\mu_{t})-u_i)-f(\zeta_i(\mu_{t})).\]

Then, we update our measures as follows, introducing some notation:

\begin{equation}
\begin{aligned}
\mus{T}&=\mus{T-}-\delta_k=\mus{T-}+\Delta_1^{S}(T-), \\
\muis{T}&=\muis{T-}+\delta_{k-(j+l+m+1)}+D(T,u,\omu{IS}{})=\muis{T-}+\Delta_1^{IS}(T-),\\
\murs{T}&=\murs{T-}+D(T,v,\murs{})=\murs{T-}+\Delta_1^{RS}(T-),\\
\muvs{T}&=\muvs{T-}+D(T,w,\muvs{})=\muvs{T-}+\Delta_1^{VS}(T-).
\end{aligned}
\end{equation}

\bigskip Another event in consideration is a recovering. Here we choose uniformly an
infected $i$ and set:
\begin{equation}
\begin{aligned}
\muis{T}&=\muis{T-}-\delta_{\zeta_i(\muis{T-})}=\muvs{T-}+\Delta_2^{IS}(T-), \\
\murs{T}&=\murs{T-}+\delta_{\zeta_i(\muis{T-})}=\murs{T-}+\Delta_2^{RS}(T-).
\end{aligned}
\end{equation}
This happens with probability $1/I_{T-}$.

\bigskip
The last event is vaccination. The corresponding rate  is $\pi_t N_t^S$. We remark the strong
dependence on the degree of the individual, because it is more probable that a higher
connectivity node to be vaccinated first. More precisely, the probability that the new
vaccinated has degree $k$ is $\frac{k\mus{T-}(k)}{N^S_{T-}}$. Once we draw the vaccinated
individual, and supposing their degree is $k$, we update the measures as follows:
\begin{equation}
\begin{aligned}
\mus{T}&=\mus{T-}-\delta_k=:\mus{T-}+\Delta_3^{S}(T-), \\
\muis{T}&=\muis{T-}+D(T,u,\muis{})=:\muis{T-}+\Delta_3^{IS}(T-),\\
\murs{T}&=\murs{T-}+D(T,v,\murs{})=:\murs{T-}+\Delta_3^{RS}(T-),\\
\muvs{T}&=\muvs{T-}+\delta_{k-(j+l+m)}+D(T,w,\muvs{})=:\muvs{T-}+\Delta_3^{VS}(T-).
\end{aligned}
\end{equation}

Now we introduce three Poisson Point Measures that will be very useful to describe the
$\M$-valued stochastic process $(\mu_t)_{t\geq 0}$. For a similar point of view, see
\cite{decreusefond2012large} or \cite{fournierMeleard04}.

The first one will provide us the possible instant in which an infection occurs. We define
$dN_1(s,k,\theta_1,j,l,m,\theta_2,u,\theta_3,v,\theta_4,w,\theta_5)$ as a product measure
on $\R_+\times E_1$ with $E_1= \N_0 \times \R_+ \times (\N_0)^3 \times \R \times
(\U\times\R_+)^3$, where $ds$ and $d\theta$ are Lebesgue measures and $dn$ are
counting measures on $\N_0$ or $\U$, accordingly.

The degree $k$ infected agent will be connected with $j$ infected, $l$ recovered and $m$
vaccinated agents, drawn according $u,v$ and $w$ as we explained above.

We also have $dN_2(s,i)$ on $E_2=\R_+\times \N$ a PPM with intensity $\gamma$ for the
recovering process. That is, for each atom we have associated a possible recovering time $s$
and the identification number $i$ of the new recovered.

The last PPM, $dN_3(s,k,\theta_1,j,l,m,\theta_2,u,\theta_3,v,\theta_4,w,\theta_5)$ is
defined in $\R_+\times E_3$ where $E_3=E_1$ and it is very similar to the first one. It assigns
a mass to each possible time $s$ where a degree $k$ vaccinated agent is connected with $j$
infected, $l$ recovered and $m$ vaccinated agents, drawn according $u,v$ and $w$.

In all the cases, the auxiliary variables $\theta$ are useful to take into account the rates in this
integral representation.

In order to simplify notation we will not write the dependency on the variables, and consider
the following indicator functions to represent the rates:
\begin{equation}
\begin{aligned}
I_1=&I_1(s,k,\theta_1,j,l,m,\theta_2,u,\theta_3,v,\theta_4,w,\theta_5)  \\
=& \mathds{1}_{\theta_1 \leq \lambda_{s-}(k)\mus{s-}(k)}\mathds{1}_{\theta_2 \leq p_{s-}(j,l,m\mid k-1)} \mathds{1}_{\theta_3\leq \rho(u\mid j+1,\muis{s-})} \\
 &\qquad\mathds{1}_{\theta_4\leq \rho(v\mid l,\murs{s-})}\mathds{1}_{\theta_5\leq \rho(w\mid m,\muvs{s-})}, \\
I_2=&I_2(s,i)=\mathds{1}_{i\leq I_{s-}}, \\
I_3=&I_3(s,k,\theta_1,j,l,m,\theta_2,u,\theta_3,v,\theta_4,w,\theta_5) \\
=& \mathds{1}_{\theta_1 \leq \pi_{s-}(k)\mus{s-}(k)}\mathds{1}_{\theta_2 \leq p_{s-}(j,l,m\mid k)} \mathds{1}_{\theta_3\leq \rho(u\mid j,\muis{s-})} \\
&\qquad \mathds{1}_{\theta_4\leq \rho(v\mid l,\murs{s-})}\mathds{1}_{\theta_5\leq \rho(w\mid m,\muvs{s-})}.
\end{aligned}
\end{equation}


Now that we have clarified the evolution of the measures according to the events that may occur,
we are ready to write an integral form for this evolution in terms of the Poisson Point
Measures, for example for the second coordinate
\begin{equation}
\muis{t}=\muis{0}+\int_0^t\sum_{k=1}^3\int_{E_k}\Delta_k^{IS}(s)I_kdN_kds.
\end{equation}
Doing the same for the four coordinates, we can write the system of Stochastic Differential
Equations:
\begin{equation}\label{SDE}
\omu{}{t}=\omu{}{0}+\int_0^t\sum_{k=1}^3\int_{E_k}\Delta_k(s)I_kdN_kds.
\end{equation}

\begin{proposition}

Given $\mu_{0}=(\mus{0},\muis{0},\murs{0},\muvs{0})$ and $N_1,N_2,N_3$ there exists a
unique strong solution to the system (\ref{SDE}) in the Skorokhod space
$\mathcal{D}(\R_+,(\M)^4)$.

\end{proposition}\begin{proof}
First note that all the measures are dominated by the expectation of
$\mus{0}+\muis{0}+\murs{0}+\muvs{0}$ and the supports are bounded on the positive
integers. The proof can be completed in the same way as in \cite{theseTran}.
\end{proof}

\subsection{Renormalization}
Inspired by the techniques developed in \cite{decreusefond2012large} and
\cite{fournierMeleard04} we write a renormalization of the system when the number of
individuals is $n$ and the number of edges is proportional to $n$. We observe that the
intensity of the jump process has the same order, and deduce the scaling for the fluid limit
renormalization. We prove the convergence of the solution of the finite case system of
equations to the solution of (\ref{SDE}) in the weak sense of the Skorokhod space
\cite{skorokhod}.

We consider four sequences of measures indexed by $n\in \N$,
$(\mu^{n,S})$, $(\mu^{n,IS})$, $(\mu^{n,RS})$ and $(\mu^{n,VS})$ satisfying the system of
equations (\ref{SDE}) for each $n\in\N$ with initial conditions
$\mu^{n,S}_0,\mu^{n,IS}_0,\mu^{n,RS}_0$ and $\mu^{n,VS}_0$. We associate $S^n_t,
I^n_t,R^n_t$ and $V^n_t$ the number of individuals in each state at time $t$ and denote
$\mathcal{S}_t,\mathcal{I}_t,\mathcal{R}_t,\mathcal{V}_t$ the sets of the nodes susceptible,
infected, recovered and vaccinated, respectively.

We take the scaling $\musn{t}=\frac{1}{n}\mu^S_t$ for each $t\leq 0$ and analogously, $\muisn{t}$, $\mursn{t}$ and $\muvsn{t}$. We denote $N^{(n),S}_t= \pint{ \musn{t}}{\chi}$ and
$S^{(n)}_t=\pint{\musn{t}}{\mathds{1}}$ and, accordingly, $N^{(n),IS}_t$, $N^{(n),RS}_t$,
$N^{(n),VS}_t$, $I^{(n)}_t$, $R^{(n)}_t$ and $V^{(n)}_t$.

Finally, we scale the rates and the indicator functions associated,
$\lambda^n_t(k)=rk\frac{\N^{n,IS}_t}{\N^{n,S}_t}$, and $p^n_{t}(j,l,m\mid k-1)=\frac{
{N^{n,IS}_{t}-1 \choose j-1}{N^{n,RS}_{t} \choose l}{N^{n,VS}_{t} \choose
m}{N^{n,S}_{t}-N^{n,RS}_{t}-N^{n,IS}_{t}-N^{n,VS}_{t} \choose k-1-j-l-m}}{{N^{n,S}_{t}-1
\choose k-1}}$,
 \begin{align*}
I_1^{(n)}=&I^{(n)}_1(s,k,\theta_1,j,l,m,\theta_2,u,\theta_3,v,\theta_4,w,\theta_5)\\
=& \mathds{1}_{\theta_1 \leq \lambda^n_{s-}(k)n\musn{s-}(k)}\mathds{1}_{\theta_2 \leq p^n_{s-}(j,l,m\mid k-1)} \mathds{1}_{\theta_3\leq \rho(u\mid j+1,n\muisn{s-})} \\
&\qquad \mathds{1}_{\theta_4\leq \rho(v\mid l,n\mursn{s-})}\mathds{1}_{\theta_5\leq \rho(w\mid m,n\muvsn{s-})}, \\
I_2^{(n)}=&I^{(n)}_2(s,i)=\mathds{1}_{i\leq I^n_{s-}}, \\
I^{(n)}_3=&I^{(n)}_3(s,k,\theta_1,j,l,m,\theta_2,u,\theta_3,v,\theta_4,w,\theta_5) \\
=& \mathds{1}_{\theta_1 \leq \pi_{s-}(k)\mus{s-}(k)}\mathds{1}_{\theta_2 \leq p^n_{s-}(j,l,m\mid k)} \mathds{1}_{\theta_3\leq \rho(u\mid j,n\muisn{s-})} \\ &\qquad \mathds{1}_{\theta_4\leq \rho(v\mid l,n\mursn{s-})}\mathds{1}_{\theta_5\leq \rho(w\mid m,n\muvsn{s-})}.
\end{align*}

We assume that the sequences of initial conditions converge weakly in $\M$ to
$\mus{0},\muis{0},\murs{0}$ and $\muvs{0}$ when $n$ goes to infinity.

We obtain the renormalized system

\begin{equation}\label{SDEn}
\mu^{(n)}_{t}=\mu^{(n)}_{0}+\frac{1}{n}\int_0^t\sum_{k=1}^3\int_{E_k}\Delta^{(n)}_k(s)I^{(n)}_kdN_kds.
\end{equation}

Let us define \[\Lambda_t=\sum_{k\in \mathbb{N}} \lambda^{(n)}_t(k)\musn{t}(k)
\sum_{j+l+m\leq k-1}p_t^n(j,l,m \mid k-1) \sum_{u\in \mathcal{U}}\rho(u\mid j+1,\muisn{t})
\] and \[ \Pi_t=\sum_{k\in\mathbb{N}} \pi_tk\musn{t}(k)\sum_{j+l+m\leq k}p_t^n(j,l,m \mid k)
\sum_{u\in \mathcal{U}}\rho(u\mid j,\muisn{t}).\]

\begin{proposition}
For all $f\in \mathcal{B}_b(\mathbb{N})$ and all $t\geq 0$ we have the following
decomposition \[ \pint{\muisn{t}}{f}=\sum_{k\in\mathbb{N}}f(k)\muisn{0}(k)+A_t^{(n),IS,f} +
M_t^{(n),IS,f}, \] where the finite variation is given by
\begin{equation}
\begin{aligned}
A_t^{(n),IS,f} = &\int_0^t \Lambda_s \left(f(k-(j+l+m+1))+ D_f(s,u,\muis{}) \right) ds \\
  &-\int_0^t \gamma \pint{\muisn{s}}{f}ds + \int_0^t \Pi_s D_f(s,u,\muis{}) ds, \end{aligned}
\end{equation}
and the associated martingale is square integrable with quadratic variation,
\begin{align*}
\left\langle M^{(n),IS,f} \right\rangle_t =&\frac{1}{n}\int_0^t \Lambda_s
\left(f(k-(j+l+m+1))+D_f(s,u,\muis{})\right)^2 ds  \\& + \frac{1}{n} \int_0^t \gamma \pint{\muisn{s}}{f^2}ds + \frac{1}{n} \int_0^t \Pi_s (D_f(s,u,\muis{}))^2 ds.
\end{align*}
\end{proposition}
\begin{proof}(Sketch)
We first calculate the infinitesimal generator $\mathcal{L}$ of our process, and we write the
Levy's martingale with $\phi=\pint{\mu}{f}$ and $\phi^2$. Then we apply the integration by
parts formula \cite{RevuzYor}, and identifying the martingales in the expression, we
rearrange the terms in order to get the quadratic variation. For a detailed proof see
\cite{fournierMeleard04}.
\end{proof}

Our fluid limit result may be proved in the same way as the proof of the main theorem of
\cite{decreusefond2012large} but we add it for completeness.
\begin{proof}[Proof of Theorem \ref{mainTh}]
Let us consider, for any $\varepsilon \leq 0$ and $A>0$, the closed set of $\M$,\[\mathcal{M}_{\varepsilon,A}={\mu \in \M : \pint{\mu}{1+\chi^5}\leq A \text{ and } \pint{\mu}{\chi}\geq \varepsilon},\] and $\mathcal{M}_{0+,A}=\bigcup_{\varepsilon>0}\mea$. Some topological properties of this can be found in the appendix of \cite{decreusefond2012large} or \cite{fournierMeleard04}.

We suppose that $\mu_0^{(n)}$ converges to $\mu_0$ and that $\mu^{(n)}_0 \in \moa^4$ for any $n$, with $\pint{\mu_0^{IS}}{\chi}>0$.

We also make the assumption $\pint{\pi_t\mu_t}{1} \leq \nu \sum_k\mu^S_0(k)\xi(k)< C \langle \mu^S_0,\chi^j\rangle$ for $j=3$ just in order to keep the hypothesis of finite $5$-th. moment, although we can change it to finite $j+2$-th. moment if necessary imposing more restrictive conditions but deriving the result making the changes needed.

In order to prove (ii), since $\lim_{\varepsilon' \to 0}t_{\varepsilon'}=\infty$, is enough to
prove the result in $\mathcal{D}([0,t_{\varepsilon'}],\mathcal{M}_{0,A}^4)$ for $\varepsilon'$
sufficiently small. From now on, we take $0< \varepsilon < \varepsilon'< \left\langle
\muis{0},\chi \right\rangle$.

\textit{Step 1: Tightness of the renormalization.} Take $(\mun{})_{n\in \N}$, $t\in \R_{>0}$
and $n\in \N$. By assumptions, we have:
\begin{equation}
\begin{array}{c}
\pint{\musn{t}}{1+\chi^5}+\pint{\muisn{t}}{1+\chi^5}+\pint{\mursn{t}}{1+\chi^5}+\pint{\muvsn{t}}{1+\chi^5}\leq \\ \leq \pint{\musn{0}}{1+\chi^5}+\pint{\muisn{0}}{1+\chi^5}\leq 2A
\end{array}
\end{equation}

This implies that the sequence $\mu_t^{(n)}$ is tight for each $t$. By the criterion of
convergence of measure valued processes proposed by Roelly \cite{roelly1986criterion} we
have to prove that, for each test function $f\in \mathcal{B}_b(\N)$, \newline
$\left(\pint{\musn{}}{f},\pint{\muisn{}}{f},\pint{\mursn{}}{f},\pint{\muvsn{}}{f}\right)_{n\in\N}$ is tight in
$\mathcal{D}(\R_{>0},\R^4)$.

We present here the calculations only for $\pint{\muisn{}}{f}$ because the others are similar
or simpler. Since we have a semimartingale decomposition, applying the Rebolledo criterion
for weak convergence of sequences of semimartingales, we have to prove that both the finite
variation part, and the quadratic variation satisfy the Aldous criterion. We want to prove that,
for all $\theta >0$ and $\eta >0$ there exist $n_0\in \N$ and $\delta >0$ such that for all
$n>n_0$ and for all stopping times $S_n$ and $T_n$ with $S_n<T_n<S_n+\delta$ we have
\begin{eqnarray}
\label{varPart}P(|A_{T_n}^{(n),IS,f}-A_{S_n}^{(n),IS,f}|>\eta)\leq \theta, \\
\label{martPart}P(|\langle M^{(n),IS,f}\rangle_{T_n}-\langle M^{(n),IS,f}\rangle_{S_n}|>\eta)\leq \theta.
\end{eqnarray}
For the finite variation condition (\ref{varPart}), we take the following bound:
\begin{align*}
& \E{|A_{T_n}^{(n),IS,f}-A_{S_n}^{(n),IS,f}|} \leq \E{\int_{S_n}^{T_n} \gamma \ninf \pint{\muisn{s}}{\mathbf{1}}ds} \\
& \qquad + \E{\int_{S_n}^{T_n} \sumn{k}\lambda^n_s(k)\musn{s}(k)\sum_{j+l+m\leq k-1}p_s^n(j,l,m|k-1)(2j+1)\ninf ds} \\
& \qquad + \E{\int_{S_n}^{T_n}\sumn{k}\pi^n_s(k)\musn{s}(k)\sum_{j+l+m\leq k}p_s^n(j,l,m|k-1)2j\ninf ds}.
\end{align*}
Since $\sum_{j+l+m\leq k}p_s^n(j,l,m|k-1)2j$ is   twice the mean number of edges with the
infected population conditioned to having degree $k$, this number is bounded by $k$, and
using the definitions of $\lambda^n$, $\pi$ and $p$ we have that:
\begin{equation}
\begin{aligned}
\E|A_{T_n}^{(n),IS,f}-&A_{S_n}^{(n),IS,f}| \leq \;
\delta \mathbb{E} \left[ \gamma \ninf (S_0^{(n)}+I_0^{(n)}) \right. \\
& \; \left. + r\ninf \pint{\musn{0}}{2\chi^2+3\chi}+\nu\ninf\pint{\musn{0}}{2\chi^2} \right] <\infty.
\end{aligned}
\end{equation}
Then, applying Markov's inequality:
\[P(|A_{T_n}^{(n),IS,f}-A_{S_n}^{(n),IS,f}|>\eta)\leq \frac{(2\gamma+5r+2\nu)\ninf \delta A}{\eta}\]
which is smaller than $\theta$ if $\delta$ is small enough.

We bound the quadratic variation of the martingale reasoning analogously,
\begin{align*}
\E{|\langle M^{(n),IS,f}\rangle_{T_n}-\langle M^{(n),IS,f}\rangle_{S_n}|} \leq &\;
 \E{\frac{\delta \gamma \ninf^2 (S_0^{(n)}+I_0^{(n)})}{n}} \\ &
 + \E{\frac{\delta r \ninf^2 \pint{\musn{0}}{\chi(2\chi+3)^2}}{n}} \\ &
 +\E{\frac{\delta \nu \ninf^2 \pint{\musn{0}}{\chi^3}}{n}} \\ \leq & \; \frac{(25r+2\gamma +\nu)\delta \ninf^2 A}{n},
\end{align*}
and applying Markov in the same way, we have the condition for the martingale part
(\ref{martPart}), so we are in the hypothesis of Aldous-Rebolledo criterion. Therefore, we
have proved tightness in $\mathcal{D}(\R_+,\mathcal{M}_{0,4}^4)$.

We now prove the uniqueness of the solution. Before, observe that, by Step 1 and
Prohorov's theorem, the laws of $\mu^{(n)}$ for $n\in \N$ are a family of bounded measures,
a precompact set in $\mathcal{D}(\R_+,\mathcal{M}_{0,4}^4)$. Hence, also are the laws of
the stopped processes $(\mu^{(n)}_{\cdot \wedge \tau^n_\varepsilon})_{n\in\N}$.

Let $\mu$ be a limit point in $\mathcal{C}(\R_+,\mathcal{M}_{0,4}^4)$ of the sequence of
stopped processes and let $(\mu^{(n)})_{n\in\N}$ be a subsequence that converges to
$\mu$, denoted with only one over-script to simplify notation. Since the limit is continuous,
the convergence is uniform over compact sets of the positive reals.

Define, for all $t\in\R_+$ and $f\in \mathcal{C}_b(\N)$ the applications
$$\Psi_t^{S,f},\Psi_t^{IS,f},\Psi_t^{RS,f},\Psi_t^{VS,f}:\mathcal{D}(R_+,\mathcal{M}_{0,4}^4)\to \mathcal{D}(\R_+,\R)$$ such that (\ref{GenEq}) can be read as
\begin{equation}\label{defPhi}
(\pint{\mus{t}}{f},\pint{\muis{t}}{f},\pint{\murs{t}}{f},\pint{\muvs{t}}{f})=(\Psi_t^{S,f}(\mu),\Psi_t^{IS,f}(\mu),\Psi_t^{RS,f}(\mu),\Psi_t^{VS,f}(\mu))
\end{equation}

\bigskip

\textit{Step 2: Uniqueness of the solution in
$\mathcal{C}(\R_+,\mathcal{M}_{0,4}\times\mathcal{M}_{0+,4}\times\mathcal{M}_{0,4}\times\mathcal{M}_{0,4})$.}

The second step consists in proving the limit values are the unique solution of (\ref{GenEq}).
The strategy will be to prove that the total measure and the first and second moments of two
solutions are equals and then prove that the generating functions of those measures
satisfies a partial differential equation that admits an unique solution in a weak sense.

Due to extension by regularity, it is enough to prove the uniqueness in \newline
$\mathcal{C}([0,T],\mathcal{M}_{0,4}\times\mathcal{M}_{\varepsilon,4}\times\mathcal{M}_{0,4}\times\mathcal{M}_{0,4})$
for all $\varepsilon,T>0$.

Take $\omu{i}{}=(\omu{S,i}{},\omu{IS,i}{},\omu{RS,i}{},\omu{VS,i}{})$ for $i=1,2$ two solutions
of (\ref{GenEq}) in this space with the same initial condition and define
\begin{equation}
\begin{aligned}
\Upsilon_t=&\sum_{j=0}^3 \vert \pint{\omu{S,1}{t}}{\chi^j}-\pint{\omu{S,2}{t}}{\chi^j}\vert \\ +&\sum_{j=0}^2 \Big( \vert \pint{\omu{IS,1}{t}}{\chi^j}-\pint{\omu{IS,2}{t}}{\chi^j}\vert
+\vert \pint{\omu{RS,1}{t}}{\chi^j}-\pint{\omu{RS,2}{t}}{\chi^j}\vert+ \\ & \qquad \qquad + \vert \pint{\omu{VS,1}{t}}{\chi^j}-\pint{\omu{VS,2}{t}}{\chi^j}\vert \Big)
\end{aligned}	
\end{equation}
Note that, for all $t\in[0,T)$ and $i=1,2$, we have $N^{S,i}_t\geq N^{IS,i}_t>\varepsilon$ and
therefore
\begin{equation}\label{pbound}
\vert \p{I,1}{t}-\p{I,2}{t}\vert \leq \frac{A}{\varepsilon^2}\vert\pint{\omu{S,1}{t}}{\chi}-\pint{\omu{S,2}{t}}{\chi}\vert +\frac{1}{\varepsilon}\vert\pint{\omu{IS,1}{t}}{\chi}-\pint{\omu{IS,2}{t}}{\chi}\vert \leq \frac{A}{\varepsilon^2}\Upsilon_t
\end{equation}
Analogously, a similar bound holds for $\vert \p{I,1}{t}-\p{I,2}{t}\vert$.

Since $\omu{i}{}$ are solutions of (\ref{GenEq}), we have, for $j=0,...,3$ and taking $f=\chi^j$
\begin{equation}
\begin{aligned}
\vert \pint{\omu{S,1}{t}}{\chi^j}-\pint{\omu{S,2}{t}}{\chi^j}\vert=&\vert \sumn{k}\mus{0}k^j(\alpha^1_t -\alpha^2_t)\vert \\
&\leq r\nu\sumn{k}k^j\mus{0}\int_0^t\vert \p{I,1}{s}-\p{I,2}{s}\vert ds \leq r\nu\frac{A^2}{\varepsilon^2} \int_0^t \Upsilon_s ds.
\end{aligned}
\end{equation}

One can reproduce similar computations for the other quantities and we get
\[\Upsilon_t \leq C(r,\gamma,\nu,A,\varepsilon)\int_0^t\Upsilon_s ds \]
that is, $\Upsilon$ satisfies a Gronwall type inequality which implies that is identically 0 for
all $t\leq T$. Then, for all $t<T$ and for $j=1,2,3$, we have
\begin{equation}
\pint{\omu{S,1}{t}}{\chi^j}=\pint{\omu{S,2}{t}}{\chi^j} \text{   ,   } \pint{\omu{IS,1}{t}}{\chi^j}=
\pint{\omu{IS,2}{t}}{\chi^j} \text{   and   }   \pint{\omu{RS,1}{t}}{\chi^j}=\pint{\omu{RS,2}{t}}{\chi^j}.
\end{equation}
This implies $\p{S,1}{t}=\p{S,2}{t}$, $\p{IS,1}{t}=\p{IS,2}{t}$,$\p{RS,1}{t}=\p{RS,2}{t}$ and
$\p{VS,1}{t}=\p{VS,2}{t}$. From the first equation in (\ref{GenEq}) and the regularity of the
solutions, we have almost sure uniqueness for $\omu{S}{}$.

It remains to prove the uniqueness for the other 3 measures. The method that we will use to
prove $\omu{IS,1}{}=\omu{IS,2}{}$ can be used for the rest.

We consider the generating functions \[\mathcal{G}^i_t(\eta)=\sum_{k\geq
0}\eta^k\omu{IS,i}{t}(k) \] for any $t\in\R_+$, $i=1,2$ and $\eta \in [0,1)$.

Let us define \[H(t,\eta)=\int_0^t\sum_{k\in \N}r\p{I,i}{s}k\;\sum_{\mathclap{\substack{\iota,j,l,m / \\ \iota+j+l+m=k-1}}}
\quad {k-1 \choose \iota,j,l,m} (\p{S}{s})^\iota(\p{I}{s})^j(\p{R}{s})^l(\p{V}{s})^m \mus{s}(k)
\eta^\iota ds\] and \[K_t=\sum_{k\in \mathbb{N}} \left[
rk\p{I}{t}(1+(k-1)\p{I}{t})+\pi_t(k)k\p{I}{t}\right]\frac{\mus{t}(k)}{N_t^{IS}}.\]

Using $f(k)=\eta^k$ in the second equation of (\ref{GenEq}) and after some basic
computations we get
\[ \mathcal{G}^i_t(\eta)= \mathcal{G}^i_0(\eta)+H(t,\eta)+\int_0^t K_s(1-\eta) \partial_\eta
\mathcal{G}^i_s(\eta)-\gamma\mathcal{G}^i_s(\eta) ds. \]

Now, $H(t,\eta)$ is continuously differentiable with respect to time and it is well defined and
bounded in $[0,T]$; and $K_t$ is  piece-wise continuous in $L^1$ and also it is well defined
and bounded on $[0,T]$. Further, $H$ and $K$ do not depend on the solution we choose,
because we already have $\omu{S,1}{}=\omu{S,2}{}$ and $\p{I,1}{}=\p{I,2}{}$. So, the
applications $t\to \tilde{\mathcal{G}}^i_t(\eta):=\mathcal{G}^i_t(\eta)e^{\gamma t}$ for
$i\in\{1,2\}$ are solutions of the equation
\[\partial_t G(t,\eta)-(1-\eta)K_t\partial_\eta G(t,\eta)=\partial_t H(t,\eta)e^{\gamma t}.\]

In view of the regularity of $H$ and $K$ it is known that this equation admits only one
solution in a weak sense (see last section in \cite{dipernalions}), hence
$\mathcal{G}^1_t(\eta)=\mathcal{G}^2_t(\eta)$ for all $t \in [0,T]$ and for all $\eta \in
[0,1)$. Since both measures have the same mass, we have $\omu{IS,1}{}=\omu{IS,2}{}$.

We can use similar arguments in order to prove that $\omu{VS,1}{}=\omu{VS,2}{}$ and
$\omu{RS,1}{}=\omu{RS,2}{}$.

\bigskip

 \textit{Step 3: $\mu^{(n)}$ satisfies asymptotically the deterministic system
(\ref{GenEq}).}

Let us remember that, for each $f\in \mathcal{C}_b(\N)$, we can write:  \[
\pint{\muisn{t}}{f}=\sum_{k\in\mathbb{N}}f(k)\muisn{0}(k)+A_t^{(n),IS,f} + M_t^{(n),IS,f}. \]

In order to characterize the limiting values, for each $n\in\N$ and for all $t \geq 0$, we have
\begin{equation}
\pint{(\mu^{(n),IS}_{t \wedge \tau^n_\varepsilon})}{f}=\Psi^{IS,f}_{t \wedge \tau^n_\varepsilon}(\mu^{(n)}+\Delta_{t \wedge \tau^n_\varepsilon}^{n,f}+M^{(n),IS,f}_{t \wedge \tau^n_\varepsilon},
\end{equation}
where $\Delta_{\cdot \wedge \tau^n_\varepsilon}^{n,f}$ vanishes in probability and
uniformly in $t$ over compact time intervals.

We can take bounds in a similar way as in Step 1 in order to get that:
\begin{equation}
\E{(M_t^{(n),IS,f})^2}=\E{\langle M^{(n),IS,f} \rangle_t}\leq \frac{(25r+\nu+2\gamma)A\ninf^2t}{n},
\end{equation}
which implies the sequence $(M_t^{(n),IS,f})_{n\in\N}$ vanishes in probability and in $L^2$,
and therefore in $L^1$ by Cauchy-Schwartz.

On the other hand, the finite variation part can be split in two: one considering the simple
edges between the newly infected node and the infected population, and a second part
regarding multiple edges, that we know is expected to vanish as the size of the population
grows. Formally,
\[ A^{(n),IS,f}_t=B^{(n),IS,f}_t+C^{(n),IS,f}_t,\]
where
\begin{equation}
\begin{aligned}
B^{(n),IS,f}_t=& \int_0^t \sumn{k}\lambda^{n}_s(k)\musn{s}(k)\quad \sum_{\mathclap{j+l+m\leq k-1}}\quad p_s^n(j,l,m \mid k-1) f(k-(j+l+m+1)) \\
 &+ \; \sum_{\mathclap{\substack{ u\in \mathcal{U}(\muisn{s},j+1) \\
u_i\leq 1}}} \; \rho(u\mid j+1,\muisn{s}) \sum_{i\leq I_{s-}^n}
[f(\zeta_i(\muisn{s})-u_i)-f(\zeta_i(\muisn{s}))]   \\
 &-  \gamma \pint{\muisn{s}}{f}ds  \\
  &+\int_0^t \sum_{k\in\mathbb{N}} \pi_sk\musn{s}(k)\sum_{j+l+m\leq k}p_s^n(j,l,m \mid k) \quad \sum_{\mathclap{\substack{ u\in \mathcal{U}(\muisn{s},j+1) \\ u_i\leq 1}}} \quad
 \rho(u\mid j,\muisn{s}) \\
 & \quad \times \left(\sum_{i=1}^{I_s}f(\zeta_i(\muisn{s})-u_i)-f(\zeta_i(\muisn{s}))\right)ds,
\end{aligned}
\end{equation}
and
\begin{equation}
\begin{aligned}
C^{(n),IS,f}_t=&\int_0^t \sumn{k}\lambda^{n}_s(k)\musn{s}(k) \sum_{j+l+m\leq k-1}p_s^n(j,l,m \mid k-1) \\
&\times \sum_{\mathclap{\substack{ u\in \mathcal{U}(\omu{n,IS}{s},j+1) \\i\leq I_{s-}^n ; \exists i\leq I^n_{s-} : u_i > 1  }}}\rho(u\mid j+1,\omu{n,IS}{s})  \left( f(\zeta_i(\omu{n,IS}{s-})-u_i)-f(\zeta_i(\omu{n,IS}{s-})) \right) ds \\
 & +\int_0^t \sum_{k\in\mathbb{N}} \pi_sk\musn{s}(k)\sum_{j+l+m\leq k}p_s^n(j,l,m \mid k) \sum_{\mathclap{\substack{
  u\in \mathcal{U}(\omu{n,IS}{s},j+1) \\
\exists i \leq I^n_{s-}: u_i > 1 }}}\rho(u\mid j,\omu{n,IS}{s}) \\
 & \quad \times \left(\sum_{i=1}^{I_{s-}}f(\zeta_i(\omu{n,IS}{s-})-u_i)-f(\zeta_i(\omu{n,IS}{s-}))\right)ds.
\end{aligned}
\end{equation}

In order to prove $C^{(n),IS,f}_t$ vanishes, let us denote
 \[ q_{j,l,s}^n=\sum_{\begin{scriptsize}\begin{array}{c} u\in \mathcal{U}(\muisn{s},j+1) \\
\exists i \leq I^n_{s-}: u_i > 1 \end{array}\end{scriptsize}}\rho(u\mid j,\omu{n,IS}{s})\] the
probability that the newly infected has a multiple edge with another infected agent. Given an
infected agent $i$, this probability is lower than the number of pairs of edges connecting the
newly infected with $i$ times the probability that these two edges in particular connecting $i$
with a susceptible individual at time $s-$ begin to connect their with the newly infected. That
is,
\begin{equation}
\begin{aligned}
q_{j,m,l,s}^n &\leq {j-1 \choose 2}\sum_{i=1}^{I^n_{s-}}\frac{k_i^{S^n_{s-}}(k_i^{S^n_{s-}}-1)}{N^{n,IS}_{s-}(N^{n,IS}_{s-}-1)}  \\
&={j-1 \choose 2}\frac{\frac{1}{n}\pint{\muisn{s-}}{\chi^2-\chi}}{N^{(n),IS}_{s-}(N^{(n),IS}_{s-}-\frac{1}{n})} \\
& \leq {j-1 \choose 2}\frac{1}{n}\frac{A}{\varepsilon(\varepsilon-\frac{1}{n})},
\end{aligned}
\end{equation}
as long as $s\leq \tau_\varepsilon^n$ and $n\geq 1/\varepsilon$. Let us remember that $k_i^{S^n_t}$ is the number of edges of the $i$-th node connecting with the susceptible population at time $t$, for a population of size $n$.

Additionally, for all the possible draws $u\in\U(j,\omu{n,IS}{s})$ we have \[
\left\vert\sum_{i=1}^{I_s}f(\zeta_i(\omu{n,IS}{s})-u_i)-f(\zeta_i(\omu{n,IS}{s}))\right\vert \leq
2j\ninf\] and applying both inequalities, for $n\geq 1/\varepsilon$ we get
\begin{equation}
\begin{aligned}
C^{(n),IS,f}_{t\wedge \tau^n_\varepsilon}&\leq \int_0^{t\wedge \tau^n_\varepsilon}
\sumn{k}rk\musn{s}(k) \\ &\qquad \qquad\times\sum_{j+l+m\leq k-1}p_s^n(j,l,m \mid k-1) 2(j+1)\ninf  \frac{Aj(j-1)}{2n\varepsilon(\varepsilon-\frac{1}{n})} ds \\
 &\quad + \int_0^{t\wedge \tau^n_\varepsilon} \sumn{k}\pi_s(k)\musn{s}(k) \\ &\qquad\qquad\times \sum_{j+l+m\leq k}p_s^n(j,l,m \mid k) 2j\ninf  \frac{A(j-1)(j-2)}{2n\varepsilon(\varepsilon-\frac{1}{n})} ds \\
 &\leq \frac{\nu r A \ninf t}{n\varepsilon(\varepsilon-\frac{1}{n})}\pint{\musn{0}}{\chi^4}.
\end{aligned}
\end{equation}

This last expression tends to zero because of the weak convergence of $\musn{0}$ to
$\omu{S}{0}$ and $\musn{s}\leq\musn{0}$ for all $s\geq 0$ and $n\in\N$.

The next task is to prove that $B^{(n),IS,f}_{\cdot \wedge \tau^n_\varepsilon}$ is similar in
some way to $\Psi^{IS,f}_{\cdot \wedge \tau^n_\varepsilon}(\mu^{(n)})$. For this, we realize
that
\begin{equation}
\begin{aligned}
&\sum_{\begin{scriptsize}\begin{aligned} &u\in \mathcal{U}(\muisn{s},j+1) \\
&\forall i\leq I_{s-}^n, u_i\leq 1 \end{aligned}\end{scriptsize}}\rho(u\mid j+1,\muisn{s}) \sum_{i\leq I_{s-}^n} \left( f(\zeta_i(\muisn{s})-u_i)-f(\zeta_i(\muisn{s}))\right) \\
&=\sum_{\begin{scriptsize}\begin{aligned} u \in (I^n_{s-})^{j+1}& \\
u_0\neq ... \neq u_j & \end{aligned}\end{scriptsize}} \left( \frac{\prod_{\iota=0}^j k_{u_\iota}(S_{s-}^n)}{N_{s-}^{n,IS}...(N_{s-}^{n,IS}-(j+1))}\right) \sum_{m=0}^j f(k_{u_m}^{S_{s-}^n}-1)-f(k_{u_m}^{S_{s-}^n}) \\
&=\sum_{m=0}^j \left( \sum_{i=1}^{I^n_{s-}} \frac{k_i^{S_{s-}^n}}{N_{s-}^{n,IS}}f(k_i^{S_{s-}^n}-1)-f(k_{i}^{S_{s-}^n} \right)\\
&\qquad\times \left(\sum_{\begin{scriptsize}\begin{aligned} u \in (I^n_{s-}\setminus\{x\})^{j} \\
u_0\neq ... \neq u_{j-1} & \end{aligned}\end{scriptsize}} \frac{\prod_{\iota=0}^{j-1} k_{u_\iota}^{S_{s-}^n}}{(N_{s-}^{n,IS}-1)...(N_{s-}^{n,IS}-(j+1))}\right) \\
&=(j+1)\frac{\pint{\muisn{s-}}{\chi(\tau_1f-f)}}{N^{n,IS}_{s-}}(1-q^n_{j,m,l,s}),
\end{aligned}
\end{equation}
where $\tau_jf(k):=f(k-j)$ for all $f:\N\to\R$ and $\forall k \in \N$.

Now we introduce some notation for the proportions of edges the newly infected agent has,
discarding the edge involved in the infection process. It is important here to make a
difference between the term that comes from the infection from the one that comes from the
PPM modeling the vaccination process, because in this case we do not assume a priori that
there is at least one infected neighbor. We define, for each $t>0$ and $n \in \N_0$,
\begin{equation}
\begin{aligned}
p_t^{n,I}&=\frac{\pint{\mu_t^{n,IS}}{\chi}-1}{\pint{\mu_t^{n,S}}{\chi}-1}, \\
p_t^{n,R}&=\frac{\pint{\mu_t^{n,RS}}{\chi}}{\pint{\mu_t^{n,S}}{\chi}-1}, \\
p_t^{n,V}&=\frac{\pint{\mu_t^{n,VS}}{\chi}}{\pint{\mu_t^{n,S}}{\chi}-1}, \\
p_t^{n,S}&=\frac{\pint{\mu_t^{n,S}}{\chi}-\pint{\mu_t^{n,IS}}{\chi}-\pint{\mu_t^{n,RS}}{\chi}-\pint{\mu_t^{n,VS}}{\chi}}{\pint{\mu_t^{n,S}}{\chi}-1}.
\end{aligned}
\end{equation}

Let us remember that
\begin{equation}
\begin{aligned}
p^n_{t}(j,l,m\mid k-1)=&\frac{ {N^{n,IS}_{t}-1 \choose j}{N^{n,RS}_{t} \choose l}{N^{n,VS}_{t} \choose m}{N^{n,S}_{t}-N^{n,RS}_{t}-N^{n,IS}_{t}-N^{n,VS}_{t} \choose k-1-j-l-m}}{{N^{n,S}_{t}-1 \choose k-1}} \\
q^n_{t}(j,l,m\mid k)=&\frac{ {N^{n,IS}_{t} \choose j}{N^{n,RS}_{t} \choose l}{N^{n,VS}_{t} \choose m}{N^{n,S}_{t}-N^{n,RS}_{t}-N^{n,IS}_{t}-N^{n,VS}_{t} \choose k-j-l-m}}{{N^{n,S}_{t} \choose k}}
\end{aligned}
\end{equation}

In the case of the infection process, we also define, for all the $j,l,m$ such that $j+l+m\leq
k-1$, and for all $n\in\N$, \[ \tilde{p}^n_{t}(j,l,m\mid
k-1)=\frac{(k-1)!\; (p_t^{n,I})^j(p_t^{n,R})^l (p_t^{n,V})^m(p_t^{n,S})^{k-1-j-l-m}}{j!\;l!\;m!\;(k-1-j-l-m)!},\]
and for the vaccinations,
\[\qtilde=\frac{k!\; (p_t^{n,I})^j(p_t^{n,R})^l(p_t^{n,V})^m(p_t^{n,S})^{k-j-l-m}}{j!\;l!\;m!\;(k-j-l-m)!}
\] the probabilities of the multinomial  variables counting the quantities of each types of
neighbors that will has the newly infected or vaccinated, respectively.

We can write
\begin{equation}
\vert B^{(n),IS,f}_{t \wedge \tau^n_\varepsilon}-\Psi^{IS,f}_{t \wedge \tau^n_\varepsilon}(\mu^{(n)})
\vert \leq \vert D^{(n),IS,f}_{t \wedge \tau^n_\varepsilon}\vert + |E^{(n),IS,f}_{t \wedge \tau^n_\varepsilon}|,
\end{equation}
where
\begin{equation}
\begin{aligned}
D^{(n),IS,f}_{t \wedge \tau^n_\varepsilon}=&\int_0^{t \wedge \tau^n_\varepsilon} \sumn{k}\lambda_s(k)^n\musn{s}(k) \\ &\qquad \times \sum_{j+l+m+1\leq k}\left[\pcomun-\ptilde\right] \\
&\qquad \times\left( f(k-(j+l+m+1))+(j+1)\frac{\pint{\muisn{s-}}{\chi(\tau_1f-f)}}{N^{n,IS}_{s-}} \right) ds \\
&+\int_0^{t \wedge \tau^n_\varepsilon} \sumn{k}\pi_s(k)^n\musn{s}(k)\quad \sum_{\mathclap{j+l+m\leq k}}\;\Big[\qcomun-\qtilde\Big] \\
&\hspace{120pt} \times\left( j\frac{\pint{\muisn{s-}}{\chi(\tau_1f-f)}}{N^{n,IS}_{s-}} \right) ds, \\
E^{(n),IS,f}_{t \wedge \tau^n_\varepsilon}=&\int_0^{t \wedge \tau^n_\varepsilon} \sumn{k}\lambda_s(k)^n\musn{s}(k)\times \sum_{j+l+m+1\leq k}\pcomun\\
&\hspace{80pt} \times\left( (j+1)\frac{\pint{\muisn{s-}}{\chi(\tau_1f-f)}}{N^{n,IS}_{s-}}q_{j-1,l,m,s}^n \right) ds \\
&+\int_0^{t \wedge \tau^n_\varepsilon} \sumn{k}\pi_s(k)^n\musn{s}(k)\times \sum_{j+l+m\leq k}\qcomun \\
&\hspace{120pt} \times\left( j\frac{\pint{\muisn{s-}}{\chi(\tau_1f-f)}}{N^{n,IS}_{s-}}q_{j,l,m,s}^n \right) ds.
\end{aligned}
\end{equation}

Thus, if we consider the differences
\[ \alpha_t^n(k)=\sum_{j+l+m+1\leq k}|\pcomun-\ptilde| \] and \[ \gamma_t^n(k)=\sum_{j+l+m\leq k}|\qcomun-\qtilde|, \] we can bound:
\begin{equation}
\begin{aligned}
|D^{(n),IS,f}_{t \wedge \tau^n_\varepsilon}|\leq \int_0^{t \wedge \tau^n_\varepsilon}
\sumn{k}\Big( rk\musn{s}(k)&\alpha_s^n(k)(1+2k)\ninf+ \\
&\pi_s(k)\musn{s}(k)\gamma_s^n(k)k\ninf \Big)ds.
\end{aligned}
\end{equation}
Since the multinomial term is a good approximation of the multivariate hypergeometric as
$n$ goes to infinity, the last expression tends to zero due to dominated convergence.

On the other hand,
\begin{equation}
\begin{aligned}
|E^{(n),IS,f}_{t \wedge \tau^n_\varepsilon}| &\leq \int_0^{t \wedge \tau^n_\varepsilon} \sumn{k}\left( (r+\nu)k^2\musn{s}(k)\ninf2 \frac{k^2A}{2n\varepsilon(\varepsilon-1/n)}\right)ds\\ &\leq \int_0^{t \wedge \tau^n_\varepsilon} \left( (r+\nu)\pint{\musn{s}}{\chi^4}\ninf \frac{A}{n\varepsilon(\varepsilon-1/n)}\right)ds \\
&\leq \frac{A^2t(r+\nu)}{n\varepsilon(\varepsilon-1/n)}.
\end{aligned}
\end{equation}
Putting all the bounds together, we can conclude that $\pint{\muisn{}}{f}$ converges in
probability uniformly over compact intervals.

\bigskip

 \textit{Step 4: The limit satisfies the deterministic system \eqref{GenEq}}

We are considering the sequence $(\mun{\cdot \wedge \tau^n_\varepsilon})_{n\in\N}$ and
we already proved that its limit in the closed set $\mathcal{M}_{0,A}^4$ is $\omu{}{}$, we
want to prove the same for the non stopped sequence. According to the Skorokhod
representation theorem there exists a subsequence on the same probability space of
$\omu{}{}$ whose marginal probability distributions are the same as those of the original
sequence such that $\mu$ is the almost sure limit. With an abuse of notation, we will denote
$(\mun{\cdot \wedge \tau^n_\varepsilon})_{n\in\N}$ this subsequence.

The mappings \[ \nu_\cdot:=(\nu^1_\cdot,...,\nu_\cdot^4)\mapsto
\frac{\pint{\nu_\cdot^k}{\mathds{1}}}{\sum_{j=1}^4 \pint{\nu_\cdot^j}{\mathds{1}}}\] for $k\in
\{1,2,3,4\}$ are continuous from $\C(\R_+,\moa\times\mea\times\moa\times\moa)$ in
$\C(\R_+,\R)$.

According to lemma (\ref{lemmaCont}) we have that, for $p\leq 5$, $\Phi_p:\D(\R_+,\mea)
\to \D(\R_+,\R)$ which assigns $\nu_\cdot \mapsto \pint{\nu_\cdot}{\chi^p}$ is continuous.

Using this, and that the quotient $(X_\cdot^1,X_\cdot^2)\mapsto
\frac{X_\cdot^1}{X_\cdot^2}$ from $\C(\R_+,\R)\times\C(\R_+,\R_*)$ to $\C(\R_+,\R)$ is
continuous, we deduce the continuity of $\nu_\cdot \mapsto
\frac{\pint{\nu^1_\cdot}{\chi}}{\pint{\nu^2_\cdot}{\chi}}$ from
$\C(\R_+,\moa\times\mea\times\moa\times\moa)$ in $\C(\R_+,\R)$. The same argument
holds for $\nu_\cdot \mapsto
\frac{\mathds{1}_{\{\pint{\nu^1_\cdot}{\chi}>\varepsilon\}}}{\pint{\nu^2_\cdot}{\chi}}$ over
the same spaces.

Since the mapping \[y\in\D([0,t],\R)\mapsto \int_0^t y_s ds\] is continuous, we conclude the
proof of the continuity of the application $\Psi_t^f$ defined in (\ref{defPhi}).

Applying Lemma (\ref{lemmaCont}) with $p=1$ we obtain that the process
$(N_{\cdot\wedge\tau_\varepsilon^n}^{(n),IS})_{n\in\N}$ converges in distribution to
$N^{IS}_\cdot:=\pint{\omu{IS}{\cdot}}{\chi}$, and, as the limit is continuous, the convergence
also holds in $(\D([0,T],\R_+),\Vert . \Vert_\infty)$ for all $T>0$
\cite{billingsley2013convergence}.

Since taking infimum is continuous over $\D(\R_+,\R)$ we have that \[\inf_{t\in
[0,T]}N^{IS}_t=\lim_{n\to\infty}\inf_{t\in [0,T]}N_{t\wedge\tau_\varepsilon^n}^{(n),IS}\] is
greater than or equal to $\varepsilon$ almost surely.

Let us define $\overline{t}_{\varepsilon'}=\inf\{t\in\R_+ : N^{IS}_t\leq \varepsilon'\}$. We do
not know this number to be deterministic, but we can say that:
\[\varepsilon' \leq \inf_{t\in [0,T]}N_{t\wedge\overline{t}_{\varepsilon'}}^{(n),IS}=\lim_{n\to\infty}\inf_{t\in [0,T]}N_{t\wedge\tau_\varepsilon^n\wedge\overline{t}_{\varepsilon'}}^{(n),IS}\]
Then, applying Fatou's Lemma,
\begin{equation}\label{stoppingIneq}
\begin{aligned}
1=&P(\inf_{t\in [0,\overline{t}_{\varepsilon'}]}N_{t\wedge\overline{t}_{\varepsilon'}}^{(n),IS}>\varepsilon)\leq \\
&\lim_{n\to\infty}P\left(\inf_{t\in [0,T\wedge\tau_{\varepsilon'}^n]}N_{t\wedge\tau_{\varepsilon}^n}^{(n),IS}>\varepsilon\right)=\lim_{n\to\infty}P(\tau_\varepsilon^n > T\wedge\overline{t}_{\varepsilon'})
\end{aligned}
\end{equation}
Therefore, splitting in the following way:
\begin{equation}\label{splitPsi}
\Psi_{\cdot \wedge\tau_\varepsilon^n\wedge\overline{t}_{\varepsilon'}\wedge T}^{IS,f}(\mu^{(n)})=\Psi_{\cdot \wedge\tau_\varepsilon^n\wedge T}^{IS,f}(\mu^{(n)})\mathds{1}_{\tau_\varepsilon^n\leq\overline{t}_{\varepsilon'}\wedge T}+\Psi_{\cdot \wedge\overline{t}_{\varepsilon'}\wedge T}^{IS,f}(\mu^{(n)}_{\cdot \wedge\tau_\varepsilon^n})\mathds{1}_{\tau_\varepsilon^n>\overline{t}_{\varepsilon'}\wedge T}
\end{equation}
we have, from the bounds and estimations we made in  Step 3, that $\Psi_{\cdot
\wedge\tau_\varepsilon^n\wedge T}^{IS,f}(\mu^{(n)})$ is bounded for the fourth moment of
$\mu^{(n)}$. Since $\mu^{(n)}_0\to\omu{}{0}$ and using (\ref{stoppingIneq}), the first term in
(\ref{splitPsi}) converges in $L^1$ and in probability to zero. On the other hand, the
continuity of $\Psi^{IS,f}$ in $\D(\R_+,\moa\times\mea\times\moa\times\moa)$,
$\Psi^{IS,f}(\mu^{(n)}_{\cdot \wedge\tau_\varepsilon^n})$ converges to $\Psi^{IS,f}(\mu)$
and therefore, $\Psi^{IS,f}_{\cdot \wedge\overline{t}_{\varepsilon'}\wedge
T}(\mu^{(n)}_{\cdot \wedge\tau_\varepsilon^n})$ converges to $\Psi^{IS,f}_{\cdot
\wedge\overline{t}_{\varepsilon'}\wedge T}(\mu)$. So,  this convergence and
(\ref{stoppingIneq}) implies that, the second term converges to $\Psi^{IS,f}_{\cdot
\wedge\overline{t}_{\varepsilon'}\wedge T}(\mu)$ in $\D(\R_+,\R)$.

Hence, $(\pint{\mu^{(n),IS}_{\cdot
\wedge\tau_\varepsilon^n\wedge\overline{t}_{\varepsilon'}\wedge T}}{f}-\Psi_{\cdot
\wedge\tau_\varepsilon^n\wedge\overline{t}_{\varepsilon'}\wedge
T}^{IS,f}(\mu^{(n)}))_{n\in\N}$ converges in probability to $\pint{\omu{}{\cdot
\wedge\overline{t}_{\varepsilon'}\wedge T}}{f}-\Psi_{\cdot
\wedge\overline{t}_{\varepsilon'}\wedge T}^{IS,f}(\mu)$. Recalling Step 3 again, and the
estimations done in it, we can conclude this sequence also converges in probability to zero.
Therefore, we have  that $\muis{}$ is a solution to the system (\ref{GenEq}) on the interval
$[0,\overline{t}_{\varepsilon'}\wedge T]$.

If either $\pint{\murs{0}}{\chi}>0$ or $\pint{\muvs{0}}{\chi}>0$, then we could apply similar
techniques with both. If not, the result can be immediately deduced because for all
$t\in[0,\overline{t}_{\varepsilon'}\wedge T]$, $\pint{\muisn{t}}{\chi}>\varepsilon$ and the
terms $\pcomun$ and $\qcomun$ are negligible when $l$ or $m$ are positives.

So, $\mu$ is almost surely the unique continuous solution of the deterministic system
(\ref{GenEq}) in $[0,\overline{t}_{\varepsilon'}\wedge T]$, which implies
$\overline{t}_{\varepsilon'}=t_{\varepsilon'}$ and the convergence in probability of
$(\mu^{(n)}_{\cdot\wedge\tau_\varepsilon^n})_{n\in\N}$ to $\mu$ holds, uniformly on the
interval $[0,t_{\varepsilon'}]$, due to the continuity of $\mu$.

In order to prove the convergence in the Skorokhod space, for $\eta >0$, we write:
\begin{equation}
\begin{aligned}
&P\Big(\sup_{t\in[0,t_{\varepsilon'}]}|\pint{\muisn{t}}{f}-\Psi^{IS,f}_t(\mu)|>\eta\Big) \leq \\
 & \quad P\Big(\sup_{t\in[0,t_{\varepsilon'}]}|\Psi^{IS,f}_{t\wedge\tau^n_\varepsilon}(\mu^{(n)})-\Psi^{IS,f}_t(\mu)|>\frac{\eta}{2} \; ; \;t_{\varepsilon'}\leq \tau_\varepsilon^n\Big)\\
&\qquad +P\Big(\sup_{t\in[0,t_{\varepsilon'}]}|\Delta^{n,f}_{t\wedge\tau^n_\varepsilon}+
M^{(n),IS,f}_{t\wedge\tau^n_\varepsilon}|>\frac{\eta}{2}\Big)+P(\tau^n_\varepsilon<t_{\varepsilon'}).
\end{aligned}
\end{equation}
Using the continuity of $\Psi^f$ and the uniform convergence in probability that we have
proved, the first term in the last expression converges to zero. In order to show that   the
second term vanish, we can reproduce the bounds taken in Step 2 of this proof and apply
Doob's inequality. Finally, since $P(\tau_\varepsilon^n>T\wedge\overline{t}_{\varepsilon'})\to
1$ we have  that the three terms goes to zero.

Hence, due to uniqueness proved in Step 2, the original sequence $(\mu^{(n)})_{n\in\N}$
converges.

 \bigskip

\textit{Step 5: The convergence of the other measures} What we have done for the
infected-susceptible connectivity measure can be also done  for the recovered and
vaccinated measures in much the same way. For the susceptible connectivity measure, one
can reason in the following way. If we consider the renormalized equation \ref{SDEn} and we
take limit in $n$, the sequence $(\musn{})_{n\in\N}$ converges in $\D(\R_+,\moa)$ to the
solution to the transport equation
\begin{equation}
\pint{\mus{t}}{f_t}=\pint{\mus{0}}{f_0}-\int_0^t \pint{\mus{s}}{(r\p{I}{s}+\pi_s)\chi f_s-\partial_s f_s}ds
\end{equation}
that can be solved as a function of $\p{I}{}$ and $\pi$, for any test function $f\in
\C_b^{0,1}(\N\times \R_+,\R)$ with bounded derivative respect time variable \cite{dipernalions}.

If we take
$f(k,s)=\varphi(k)e^{-\int_0^{t-s}rk\p{I}{u}+\pi_u(k)du}$ we obtain
\[\pint{\mus{t}}{\varphi}=\sumn{k} \varphi(k)\alpha_t^k \theta^{\xi(k)} \mus{0}(k)\] as the first equation of
(\ref{GenEq}) establishes.

The proof is finished.
\end{proof}

\begin{lemma}\label{lemmaCont}
For any $p\leq 5$, the map $\Phi_p:\D(\R_+,\mea)\to \D(\R_+,\R)$ that assigns $\Phi(\nu_.)
\mapsto \pint{\nu_.}{\chi^p}$ is continuous.
\end{lemma}
\begin{proof}
The proof can be obtained by following the steps of Lemma 1-5 in the appendix of
\cite{decreusefond2012large}.
\end{proof}

\begin{acknowledgements}
The authors E.J. Ferreyra and M. Jonckheere are partially supported by Universidad de Buenos Aires under grants PICT 2018-03630, by Agencia Nacional de Promoci\'on Cinet\'ifica y
T\'ecnol\'ogica.
The author J.P. Pinasco was partially supported by
Universidad de Buenos Aires under grants
20020170100445BA,  by Agencia Nacional de Promoci\'on Cinet\'ifica y
T\'ecnol\'ogica PICT2016-1022.

\end{acknowledgements}

\bibliographystyle{spmpsci}
\bibliography{bibliografiasFerreyra.bib}
\end{document}